\documentclass[12pt, reqno]{amsart}
\usepackage{amsmath, amsthm, amscd, amsfonts, amssymb, graphicx, color,
mathrsfs,latexsym,paralist, soul}
\usepackage[bookmarksnumbered, colorlinks, plainpages]{hyperref}
\usepackage{tikz}
\usetikzlibrary{arrows,automata}
\hypersetup{colorlinks=true,linkcolor=red, anchorcolor=green, citecolor=cyan, urlcolor=red, filecolor=magenta, pdftoolbar=true}

\newcommand{\shf}{\underline}

\newtheorem{thm}{Theorem}[section]
\newtheorem{prop}[thm]{Proposition}

\theoremstyle{remark}
\newtheorem{rmk}[thm]{Remark}
\theoremstyle{definition}
\newtheorem{example}[thm]{Example}
\newtheorem{defn}[thm]{Definition}

\DeclareMathOperator{\Hom}{Hom}
\DeclareMathOperator{\Ext}{Ext}

\newcommand{\z}{^{(0)}}
\newcommand{\2}{^{(2)}}
\newcommand{\inv}{^{-1}}

\newcommand{\bi}{\begin{itemize}}
\newcommand{\ei}{\end{itemize}}
\newcommand{\be}{\begin{enumerate}}
\newcommand{\ee}{\end{enumerate}}

\newcommand{\T}{\mathbb{T}}
\renewcommand{\H}{\mathcal{H}}

\newcommand{\A}{\mathcal{A}}
\newcommand{\B}{\mathcal{B}}
\newcommand{\C}{\mathcal{C}}

\newcommand{\G}{\mathcal{G}}

\renewcommand{\S}{\mathfrak{S}}

\newcommand{\N}{\mathbb{N}}
\newcommand{\Z}{\mathbb{Z}}

\newcommand{\roof}[1]{\lceil{#1}\rceil}

\begin{document}

\setcounter{page}{1}

\title[Cohomology for small categories]{Cohomology for small categories: $k$-graphs and groupoids}

\author[E. Gillaspy \MakeLowercase{and} A. Kumjian]{Elizabeth Gillaspy$^{1,*}$  \MakeLowercase{and}  Alexander Kumjian$^2$}

\address{$^{1}$Mathematisches Institut, Universit\"at M\"unster, 48149 M\"unster, Germany.}
\email{\textcolor[rgb]{0.00,0.00,0.84}{gillaspy@uni-muenster.de}}

\address{$^{2}$Department of Mathematics, University of Nevada, Reno, NV 89557, USA.}
\email{\textcolor[rgb]{0.00,0.00,0.84}{alex@unr.edu}}


\let\thefootnote\relax\footnote{Copyright 2017 by the Tusi Mathematical Research Group.}



\date{Received: xxxxxx; Revised: yyyyyy; Accepted: zzzzzz.
\newline \indent $^{*}$Corresponding author}

\begin{abstract}
Given a higher-rank graph $\Lambda$, we investigate the relationship between the cohomology of $\Lambda$ and the cohomology of the associated groupoid $\G_\Lambda$.
We define an exact functor between the abelian category of right modules over a higher-rank graph $\Lambda$ 
and the category of $\G_\Lambda$-sheaves, where $\G_\Lambda$ is the path groupoid of $\Lambda$.
We use this functor to construct compatible homomorphisms
from both the cohomology of $\Lambda$ with coefficients in a right $\Lambda$-module, and the 
continuous cocycle cohomology of  $\G_\Lambda$ with values in the corresponding $\G_\Lambda$-sheaf,
 into the sheaf cohomology of $\G_\Lambda$.
\end{abstract} 

\maketitle
\section{Introduction}

Higher-rank graphs (also called $k$-graphs) were introduced by Kumjian and Pask in \cite{kp} as a combinatorial model for certain higher rank Cuntz-Krieger $C^*$-algebras (see \cite{rob-ste}) in a construction that also generalized the definition of the $C^*$-algebra of a directed graph.
As the name suggests, a $k$-graph is often best viewed as a $k$-dimensional generalization of a directed graph.  
Formally, however,  a $k$-graph (see Definition \ref{def:kgraph} below) is defined to be a {countable} category with additional structure.
The combinatorial nature of $k$-graph $C^*$-algebras has facilitated the analysis of  structural properties of these $C^*$-algebras, such as their simplicity and ideal structure \cite{rsy2, robertson-sims, skalski-zacharias-1, davidson-yang-periodicity, kang-pask}, quasidiagonality \cite{clark-huef-sims} and KMS states \cite{aHLRS, aHLRS1}.  In particular, results such as \cite{spielberg-kirchberg,bnrsw,bcfs,prrs} show that $k$-graphs often provide concrete examples of $C^*$-algebras which are relevant to Elliott's  classification program for simple separable nuclear $C^*$-algebras.

Under mild hypotheses -- namely, that $\Lambda$ is row-finite and has no sources -- we can associate a groupoid $\G_\Lambda$ to a $k$-graph $\Lambda$, 
and it was also shown in \cite{kp} that the groupoid $C^*$-algebra $C^*(\G_\Lambda)$, as defined in \cite{renault}, is isomorphic to the $C^*$-algebra $C^*(\Lambda)$ of the $k$-graph.
More generally, given a locally compact Hausdorff groupoid $\G$ and a continuous $\T$-valued 2-cocycle $\sigma$ on $\G$,
 Renault  defines in \cite{renault} a twisted version $C^*(\G, \sigma)$ of the groupoid $C^*$-algebra, such that, up to isomorphism, $C^*(\G, \sigma)$ only depends on the class $[\sigma] \in H^2_c(\G, \T)$ of $\sigma$ in the continuous cocycle cohomology of $\G$ with coefficients in $\T$.  

Kumjian, Pask, and Sims initiated in \cite{kps-hlogy}  the study of the cohomology of a $k$-graph taking values in an abelian group $A$.  They also explained how to construct, from a $k$-graph $\Lambda$ and a $\T$-valued 2-cocycle on $\Lambda$, a twisted higher-rank-graph $C^*$-algebra $C^*(\Lambda, c)$. 
In  Lemma 6.3 of \cite{kps-twisted}, they constructed a map  $\sigma$ from the $k$-graph 2-cocycles $Z^2(\Lambda, A) $ to the continuous 2-cocycles $Z^2_c(\G_\Lambda, A)$ of the associated groupoid,
such that $\sigma$ maps coboundaries to coboundaries, and proved that $C^*(\Lambda, c) \cong C^*(\G_\Lambda, \sigma(c))$ 
when $A = \T$  (see \cite[Theorem 6.5 and Corollary 7.9]{kps-twisted}).  
However, the map $\sigma$ is \emph{ad hoc} and does not easily generalize to provide a homomorphism $H^n(\Lambda, A) \to H^n_c(\G_\Lambda, A)$ for $n \not= 2$.  

{The initial motivation for the research presented in this article was to understand more thoroughly the relationship between the cohomology of a $k$-graph $\Lambda$ and that of its associated groupoid.}

Since $k$-graphs are small categories, 
instead of the approach to the cohomology of $k$-graphs from \cite{kps-twisted}, a natural alternative is to use the cohomology theory of small categories (cf.~\cite{watts,baues-wirsching,xu}). {We review these constructions in Section \ref{sec:modules} below.}
For a $k$-graph $\Lambda$,
Proposition \ref{prop:same-coh} below shows that the usual definition of the cohomology of a small category with coefficients in a $\Lambda$-module (see Definitions \ref{def:module} and \ref{def:module-cohlogy} below) agrees with the categorical cohomology of \cite{kps-twisted} {in the case that the module is a constant abelian group}.  

The advantage of the $\Lambda$-module perspective on the cohomology of a $k$-graph $\Lambda$ is that it enables us to relate the $k$-graph cohomology to the groupoid sheaf cohomology in all degrees.
The
 first main result of this paper, Theorem \ref{thm:mod-sheaf-functor}, is the construction {(under mild hypotheses)} of an exact additive functor, $\A \mapsto \underline{\A}$, from the category of right $\Lambda$-modules to the category of $\G_\Lambda$-sheaves.  
 In Section \ref{sec:equiv-sheaf}, we then construct natural maps into the sheaf cohomology $H^n(\G_\Lambda, \shf{\A})$ 
from both the $k$-graph cohomology $H^n(\Lambda, \A)$ and the continuous cocycle cohomology $H^n_c(\G_\Lambda, \shf{\A})$.
 Moreover, we show in Proposition \ref{pr:psi-isom} that both of these maps factor through the functor $\A \mapsto \shf{\A}$.  Indeed, the second main result of the paper, Theorem 4.1, consists of showing  the commutativity of the diagram below (for $n \leq 2$):
 \[ 
\begin{tikzpicture}[>=stealth]
    \node (02) at (0,2) {$H^n(\Lambda, \A)$};
    \node (42) at (4,2) {$H^n(\Hom_{\G_\Lambda}(\shf{P^*}, \shf{\A}))$};
    \node (00) at (0,0) {$H^n_c(\G_\Lambda,  \shf{\A}))$};
    \node (40) at (4,0) {$H^n(\Hom_{\G_\Lambda}(\mathscr{P}_*, \shf{\A}))$};
    \node (82) at (8,2) {$H^n(\G_\Lambda,  \shf{\A}))$};
    \draw[->] (40) -- node[right] {${\psi_n^*}$} (42);
    \draw[->] (00) -- node[above] {${{\eta^n}}$} node[below] {${\scriptstyle{\cong}}$} (40);
    \draw[->] (02) -- (42);
    \draw[->] (02) -- node[left] {\cite{kps-twisted}}  node[right]{($n \leq 2$)} (00);
    \draw[->] (40) -- node[below] {${\rho^n_{\mathscr{P}}}$} (82);
    \draw[->] (42)-- node[above] {${\rho^n_{\shf{P}}}$} (82);    
\end{tikzpicture}
 \]
Although the left-hand vertical arrow above is not defined for $n > 2$, the right-hand triangle commutes for all $n$, and the composition 
\[\rho^n_{\mathscr{P}} \circ \eta^n: H^n_c(\G, \mathscr{A}) \to H^n(\G, \mathscr{A})\]
 is well defined for any \'etale groupoid $\G$ and any $\G$-sheaf $\mathscr{A}$; see Propositions  \ref{prop:cocycle-coh},  \ref{prop:exact-to-sheaf}  and \ref{pr:psi-isom}
below.

 We observe in Remark \ref{rmk:ample-gpoid-sheaf} that the diagonal arrow in the diagram above is an isomorphism when $n=2$.  
One might then begin to suspect that other arrows in the diagram might also be isomorphisms.  
 However, Kumjian, Pask, and Sims exhibit in Remark 6.9 of \cite{kps-twisted}  a 1-graph $B_2$ such that the left-hand vertical arrow $H^1(B_2, A) \to H^1(\G_{B_2}, A)$ is not onto.
 Example \ref{ex:not-isom} below complements this example, and answers in the negative a conjecture from \cite{kps-twisted}, by showing that the homomorphism $H^1(\Lambda, \A) \to H^1_c(\G_\Lambda, \underline{\A})$ need not be  injective.



\section{Modules over categories}
\label{sec:modules}

In this section, we present the cohomology theory for small categories which provides the foundation for 
our main results in Sections \ref{sec:kgraph-gpoid} and \ref{sec:equiv-sheaf},
relating the cohomology of a higher-rank graph with that of its associated groupoid.  Most of the material in this section is surely well known to the experts in the area,  
but we have chosen to present it in detail here both as a courtesy to our intended audience, namely those interested in higher-rank graphs and their associated $C^*$-algebras, and to have the material available in the form we require.

{To be precise, we begin by describing (in Definition \ref{def:module})
modules over small categories, and (in Definition \ref{def:proj-res}) a canonical projective resolution of the constant module $\Z^\Lambda$ over a category $\Lambda$, which we use in Definition \ref{def:module-cohlogy} to define the cohomology groups of $\Lambda$ with coefficients in a $\Lambda$-module $\A$.  We show in Proposition \ref{prop:same-coh} that the cohomology groups thus defined agree with the categorical cohomology groups employed in \cite{kps-twisted} to study the cohomology of  higher-rank graphs. }

Before we begin, a brief word about notation.  Throughout this paper, $\N := \{ 0, 1, 2, \ldots \}$ is regarded as a monoid under addition, or equivalently a category with one object.  Thus, the notation $n \in \N$ indicates that $n$ is a morphism, rather than an object, of $\N$.  Inspired by this, we use the arrows-only picture of category theory throughout this paper, so that the notation 
$\lambda \in \Lambda$ implies that $\lambda$ is a morphism in the category $\Lambda$.
\begin{defn}(cf.~\cite{watts})
\label{def:module}
Given a small category $\Lambda$, we define a \emph{(right) $\Lambda$-module} to be a contravariant functor $\A: \Lambda \to {\bf Ab}$, where ${\bf Ab}$ denotes the category of abelian groups.  For $v \in \text{Obj}(\Lambda)$, we will write  $\A_v$ for the associated abelian group, and for a morphism $\lambda \in \Lambda$, we will write $\A(\lambda): \A_{r(\lambda)} \to \A_{s(\lambda)}$ for the associated homomorphism of abelian groups. 
A \emph{morphism of $\Lambda$-modules} is then a natural transformation.  A sequence of $\Lambda$-modules 
\[ \begin{tikzpicture}
			\node (phi) [label=left:{$\A$}] at (-2,0){};
			\node (psi) [label=center:{$\B$}, outer sep = 4pt] at (0,0){};
			\node (xi) [label=right:{$\C$}] at (2,0){};
			
		\begin{scope}[->,  line width=0.5]
			\draw  (phi) to node[above] {$\eta$}  (psi); 		
			\draw  (psi) to node[above] {$\rho$}  (xi); 		
		\end{scope}
	\end{tikzpicture}
\]
is \emph{exact} if for all $v \in \text{Obj}(\Lambda)$, the sequence of abelian groups 
\[ \begin{tikzpicture}
			\node (phi) [label=left:{$\A_v$}] at (-2,0){};
			\node (psi) [label=center:{$\B_v$}, outer sep = 4pt] at (0,0){};
			\node (xi) [label=right:{$\C_v$}] at (2,0){};
			
		\begin{scope}[->,  line width=0.5]
			\draw  (phi) to node[above] {$\eta(v)$}  (psi); 		
			\draw  (psi) to node[above] {$\rho(v)$}  (xi); 		
		\end{scope}
	\end{tikzpicture}
\]
is exact.  
\end{defn}

Note that $\mathfrak{M}_\Lambda$, the category of right $\Lambda$-modules, forms an abelian category.
Given an abelian group $A$, we write $A^\Lambda$ for the $\Lambda$-module, or functor, given by $A^\Lambda_v = A$ for all $  v \in \text{Obj}(\Lambda)$ and $A^\Lambda(\lambda) = id$ for all morphisms $ \lambda \in \Lambda$.


\begin{defn}
\label{def:proj-module}
Let $\Lambda$ be a small category.  A \emph{projective $\Lambda$-module} is a $\Lambda$-module $\mathcal{P}$ such that every surjective morphism $\eta: \A \to \mathcal{P}$ of $\Lambda$-modules splits as a map of $\Lambda$-modules. 

Equivalently, $\mathcal{P}$ is projective if, for any exact sequence $\A \to \B \to 0$ of $\Lambda$-modules, the natural map $\Hom_\Lambda(\mathcal{P}, \A) \to \Hom_\Lambda(\mathcal{P}, \B)$ is onto.
\end{defn}

{Let $\{ \A_n\}_{n \in \N}$ be a sequence of $\Lambda$-modules and $\{ \eta_n: \A_{n} \to \A_{n-1}\}_{n\geq 1}$ a sequence of $\Lambda$-module morphisms.  We say that $\{ \A_n, \eta_n\}_{n\in \N}$ is a \emph{resolution} of a $\Lambda$-module $\B$  if $\ker \eta_n = \text{Im}\, \eta_{n+1}$ for all $n \geq 1$, and there also exists a  surjective morphism $\eta_0: \A_0 \to \B$ of $\Lambda$-modules such that  $\ker \eta_0 = \text{Im}\, \eta_1$. 

To define the cohomology of a small category $\Lambda$, we will use a projective resolution $\{P^n, d_n\}_{n\in \N}$ of the constant $\Lambda$-module $\Z^\Lambda$.}
The projective resolution $\{P^n, d_n\}_n$  is essentially 
a reformulation of the projective resolution $\{P_n, \sigma_n\}_n$ described on \cite{xu} page 2567. {Moreover, the cohomology groups arising from this projective resolution agree with the definitions in the literature of the categorical cohomology of a higher-rank graph or an \'etale groupoid; see Propositions \ref{prop:same-coh} and \ref{prop:cocycle-coh} below.}

\begin{defn}
\label{def:proj-res}
Let $\Lambda$ be a small category.  For $n \geq 1$, we write $\Lambda^{*n}$ for the set of composable $n$-tuples in $\Lambda$:
\[\Lambda^{*n} = \{( \lambda_1, \ldots, \lambda_n)\in \prod_{i=1}^n \Lambda: s(\lambda_i) = r(\lambda_{i+1}) \ \forall \ i \}.\]
We define $\Lambda^{*0} = \Lambda^0$ to be the objects of $\Lambda$.
For each $v \in \text{Obj}(\Lambda)$, let 
\[P^n_v = \Z \Lambda^{*(n+1)}v\] be the free abelian group generated by the collection $\Lambda^{*(n+1)}v$ of all $(n+1)$-tuples $(\lambda_0,\ldots, \lambda_n)$ of composable morphisms  with $s(\lambda_{n})=v$. 
When it is useful to distinguish an $(n+1)$-tuple $(\lambda_0, \ldots, \lambda_n)$ from the associated generator of $P^n_v$, we will write $[\lambda_0, \ldots, \lambda_n]$ for the latter. 

  For each $\lambda \in \Lambda$, there is a unique morphism $P^n(\lambda): P^n_{r(\lambda)} \to P^n_{s(\lambda)}$ such that  
  \[ 
  P^n(\lambda)([\lambda_0, \lambda_1, \ldots, \lambda_n]) = [\lambda_0, \lambda_1, \ldots, \lambda_n\lambda ].
  \]
\end{defn}

For $n \geq 1$, we define  $d_n: P^n_v \to P^{n-1}_v$ on generators $[\lambda_0, \lambda_1, \ldots, \lambda_n]$ and extend $\Z$-linearly:
\begin{align*}
d_n([\lambda_0, \lambda_1, \ldots, \lambda_n]) &=   [\lambda_1, \lambda_2, \ldots, \lambda_n] + \sum_{i=1}^{n} (-1)^{i} [\lambda_0, \lambda_1, \ldots, \lambda_{i-1} \lambda_{i}, \ldots, \lambda_n].
\end{align*}
Similarly, the morphism $d_0: P^0_v \to \Z$ is defined on generators by  $d_0([\lambda]) = 1 $.  
One 
easily checks that $d_n$ is a natural transformation for each $n$, and that $d_n \circ d_{n+1} = 0$.  

\begin{prop}
The complex $\{P^n, d_n\}_{n\in \N}$ is a projective resolution of $\Z^\Lambda$.
\label{prop:proj-res}
\end{prop}
\begin{proof}
We begin by showing that $P^n$ is projective for all $n$.  To this end, let $\eta: \A \to P^n$ be a surjective morphism of $\Lambda$-modules.
For each $n+1$-tuple of the form $(\lambda_0, \ldots, \lambda_{n-1}, s(\lambda_{n-1}))$, choose an element $a_{(\lambda_0, \ldots, \lambda_{n-1}, s(\lambda_{n-1}))} \in \A$
 such that \[\eta(a_{(\lambda_0, \ldots, \lambda_{n-1}, s(\lambda_{n-1}))})= [\lambda_0, \ldots, \lambda_{n-1}, s(\lambda_{n-1})].\]  Define a morphism $\xi: P^n \to \A$ of $\Lambda$-modules by defining $\xi$ on generators of the form $[\lambda_0, \ldots, \lambda_{n-1}, s(\lambda_{n-1})]$ by
\[\xi([\lambda_0, \ldots, \lambda_{n-1}, s(\lambda_{n-1})])= a_{(\lambda_0, \ldots, \lambda_{n-1}, s(\lambda_{n-1}))};\]
then we extend $\xi$ to all of $P^n$ so that $\xi$ becomes a morphism of $\Lambda$-modules.  In other words, if $(\lambda_0, \ldots, \lambda_n)$ is an arbitrary element of $\Lambda^{*(n+1)}$, we have 
\[\xi( [\lambda_0, \ldots, \lambda_n] ) := \A(\lambda_n)(\xi([\lambda_0, \ldots, \lambda_{n-1}, s(\lambda_{n-1})])) = \A(\lambda_n)(a_{(\lambda_0, \ldots, \lambda_{n-1}, s(\lambda_{n-1}))}).\]
By construction, $\eta \circ \xi = id$ as morphisms of $\Lambda$-modules.  Thus, $P^n$ is projective, as claimed.

Now, we show that $\{P^n, d_n\}_{n\in \N}$ is  a resolution of $\Z^\Lambda$.  Since $d_n \circ d_{n+1} = 0$ for all $n$, and $d_0: P^0 \to \Z^\Lambda$ is clearly onto, it suffices to show that $\ker d_n \subseteq \text{Im } d_{n+1}$.  

To that end, we construct a contracting homotopy $(s_n)_{n\in \N}$ where $s_n: P_n \to P_{n+1}$ is a natural transformation for each $n$.  
Fix $v \in \Lambda^0$ and define $s_{n} = s_{n, v}: P^n_v \to P^{n+1}_v$ on generators by
\[ 
s_{n}([\lambda_0, \ldots, \lambda_n]) = (-1)^{n+1} [\lambda_0, \ldots, \lambda_n, s(\lambda_n)],
\]
{and extend $\Z$-linearly.}
{We also define $s_{-1}: \Z^\Lambda_v \to P^0_v$ by $s_{-1}([1]) = [v]$.}
Although $s_n$ does not {commute with the action of $\Lambda$, and hence fails to} determine a morphism of $\Lambda$-modules, {for $n \geq 1$} we have 
\begin{align*}
(d_{n+1} \circ s_n + s_{n-1} \circ d_n )& ([\lambda_0, \ldots, \lambda_n]) = d_{n+1}( (-1)^{n+1} [\lambda_0, \ldots, \lambda_n, s(\lambda_n)]) \\
& \quad + s_{n-1}\left( [\lambda_1, \ldots, \lambda_n] + \sum_{j=1}^n (-1)^j [\lambda_0, \ldots, \lambda_{j-1} \lambda_j, \ldots, \lambda_n] \right) \\
&= (-1)^{n+1} ( [\lambda_1, \ldots, \lambda_n, s(\lambda_n)] + (-1)^{n+1} [\lambda_0, \ldots, \lambda_n]  \\
& \qquad  + \sum_{j=1}^n (-1)^j [\lambda_0, \ldots, \lambda_{j-1} \lambda_j, \ldots, \lambda_n, s(\lambda_n) ]\\
& \quad + (-1)^n [\lambda_1, \ldots, \lambda_n, s(\lambda_n)]   )\\
& \qquad + \sum_{j=1}^n (-1)^{j+n} [\lambda_0, \ldots, \lambda_{j-1} \lambda_j, \ldots, \lambda_n, s(\lambda_n)]  \\
&= [\lambda_0, \ldots, \lambda_n].
\end{align*}
{If $n=0$, we compute 
\begin{align*}
(d_1 \circ s_0 + s_{-1} \circ d_0)([\lambda]) & = -d_1([\lambda, s(\lambda)]) + [s(\lambda)] = -[s(\lambda)] + [\lambda] - [s(\lambda)] \\
&= [\lambda].
\end{align*}
Thus, if $a \in \ker d_n$, we have $a = (d_{n+1} \circ s_n + s_{n-1} \circ d_n )(a)  \in \text{Im } d_{n+1}.$}
Since $\ker d_n \subseteq \text{Im } d_{n+1}$ for all $n$, the projective complex $\{ P^n, d_n\}_{n\in \N}$ is in fact a resolution of $\Z^\Lambda$,  as claimed.
\end{proof}
\begin{defn}
\label{def:module-cohlogy}
Let $\Lambda$ be a small category and let $\A$ be a $\Lambda$-module. We define the \emph{cohomology groups of $\Lambda$ with coefficients in $\A$} to be the cohomology groups of the complex $\Hom_\Lambda(P^*, \A)$, where the boundary map $\delta_n: \Hom_\Lambda(P^n, \A) \to \Hom_\Lambda(P^{n+1}, \A)$ is $\delta_n(c) := c \circ d_{n+1}$.  
 
 In more detail, an \emph{$n$-cochain} is a morphism of $\Lambda$-modules $c: P^n \to \A$, and an \emph{$n$-cocycle} is an $n$-cochain $c$ such that $\delta_n(c) := c \circ d_{n+1} = 0$.   
We say $c$ is an \emph{$n$-coboundary} if $c = {\delta_{n-1}}(b)$ for some $b \in \Hom_\Lambda(P^{n-1}, \A)$.


{Observe that $\Hom_\Lambda(P^n, \A)$ is an abelian group for all $n$.} Writing ${Z}^n(\Lambda, \A)$ for the subgroup of $n$-cocycles, and ${B}^n(\Lambda, \A)$ for the subgroup of coboundaries,  the cohomology of $\Lambda$ with coefficients in $\A$ is 
\begin{equation}{H}^n(\Lambda, \A) := \frac{{Z}^n(\Lambda, \A)}{{B}^n(\Lambda, \A)}.
\label{eq:kgraph-coh}
\end{equation}
\end{defn}

We now compare Definition \ref{def:module-cohlogy} with 
the \emph{categorical cohomology} of $\Lambda$ introduced in Definition 3.5 of \cite{kps-twisted}.  
Although the  categorical cohomology was originally defined in a slightly more restrictive setting than the one we present in Definition \ref{defn:categorical} below, Definition \ref{defn:categorical} agrees with the above-cited Definition 3.5 when the module is a constant abelian group.
Moreover (cf. Proposition \ref{prop:same-coh}) the categorical cohomology groups of Definition \ref{defn:categorical} agree with the cohomology groups of Definition \ref{def:module-cohlogy} above.

\begin{defn}
\label{defn:categorical}
Let $\Lambda$ be a small category and let $\A$ be a $\Lambda$-module. The group of \emph{categorical $n$-cochains} on $\Lambda$ taking values in $\A$ is 
\[\tilde{C}^n(\Lambda, \A) := \{ c: \Lambda^{*n} \to \A \mid c(\lambda_1, \ldots, \lambda_n) \in \A_{s(\lambda_n)}\}.\]

We define $\tilde{\delta}^n: \tilde{C}^n(\Lambda, \A) \to \tilde{C}^{n+1}(\Lambda, \A)$ by 
\begin{align*}
\tilde{\delta^n}(c)(\lambda_0, \ldots, \lambda_{n}) &= c(\lambda_1, \ldots, \lambda_n) + \sum_{i=1}^{n} (-1)^i c(\lambda_0, \ldots, \lambda_{i-1} \lambda_i, \ldots, \lambda_n)\\ & \qquad + (-1)^{n-1} \A(\lambda_n)( c(\lambda_0, \ldots, \lambda_{n-1})) 
\end{align*}
if $n \geq 1$; we set $\tilde{\delta^0}(c)(\lambda) = c(s(\lambda)) - \A(\lambda)(c(r(\lambda))).$

A \emph{categorical $\A$-valued $n$-cocycle} on $\Lambda$ is then an element of $ \tilde{Z}^n(\Lambda, \A):= \ker \tilde{\delta}^{n+1}$; we call $\tilde{B}^n(\Lambda, \A):=\text{Im}\, \tilde{\delta}^n$ the group of \emph{categorical $\A$-valued $n$-coboundaries}. 

One easily checks (cf.~Lemma 3.4 of \cite{kps-twisted}) that $(\tilde{C}^*(\Lambda, \A), \tilde{\delta}^*)$ is a cochain complex; thus, we can define the \emph{categorical cohomology} $\tilde{H}^*(\Lambda, \A)$ of $\Lambda$ with coefficients in $\A$ to be the cohomology of this complex, namely,
\[
\tilde{H}^n(\Lambda, \A) := \tilde{Z}^n(\Lambda, \A)/\tilde{B}^{n-1}(\Lambda, \A).
\]
\end{defn}

\begin{rmk}
{In contrast to Definition \ref{def:module-cohlogy}, 
there is   no mention of equivariance in Definition \ref{defn:categorical} above, because we have not defined an action of $\Lambda$ on $\Lambda^{*n}$.  In other words, categorical cochains and cocycles are merely set maps satisfying the conditions in Definition \ref{defn:categorical}.}
We also observe that, unlike in \cite{kps-twisted}, we do not require that our cochains and cocycles be normalized.  However, every categorical cocycle (in the sense of Definition \ref{defn:categorical}) is cohomologous to a normalized cocycle, so the cohomology groups of Definition \ref{defn:categorical} agree with those of \cite{kps-twisted} Definition 3.5 in the case when $\Lambda$ is a $k$-graph (see Definition \ref{def:kgraph} below) and $\A = A^\Lambda$ for an abelian group $A$.
\end{rmk}

\begin{prop}
For any small category $\Lambda$ and any $\Lambda$-module $\A$, the map $\zeta^n: \Hom_\Lambda(P^n, \A) \to \tilde{C}^n(\Lambda, \A)$ given by 
\[\zeta^n f(\lambda_1, \ldots, \lambda_n) = f([\lambda_1, \ldots, \lambda_n, s(\lambda_n)])\]
induces an isomorphism ${H}^*(\Lambda, \A) \cong \tilde{H^*}(\Lambda, \A)$.
\label{prop:same-coh}
\end{prop}
\begin{proof}
We will show that $\zeta^n$ is an isomorphism of abelian groups by exhibiting an inverse $\eta^n: \tilde{C^n}(\Lambda, \A) \to \Hom_\Lambda(P^n, \A)$, namely,
\[\eta^n c( [\lambda_0, \ldots, \lambda_n]) = \A(\lambda_n)(c(\lambda_0, \ldots, \lambda_{n-1})).\]

To see that $\eta^n c \in \Hom_\Lambda(P^n, \A)$, we note that, whenever $r(\lambda) = s(\lambda_n)$,
\begin{align*}
\A(\lambda) \eta^n c([\lambda_0, \ldots, \lambda_n])&= \A(\lambda)(\A(\lambda_n)(c(\lambda_0, \ldots, \lambda_{n-1}))) = \eta^n c([\lambda_0, \ldots, \lambda_n \lambda]) \\
& = \eta^n c(P^n(\lambda)[\lambda_0, \ldots, \lambda_n] ).
\end{align*}

A straightforward computation shows that $\zeta^n \circ \eta^n = id$; to see that $\eta^n \circ \zeta^n = id$, we rely on the fact that  for any $f \in \Hom_\Lambda(P^n, 
\A)$ and any $(\lambda_0, \ldots, \lambda_n) \in \Lambda^{*(n+1)}$, we have 
\begin{align*}
\A(\lambda_n)(f([\lambda_0, \ldots, \lambda_{n-1}, s(\lambda_{n-1})])) &= f(P^n(\lambda_n)[\lambda_0, \ldots, \lambda_{n-1}, s(\lambda_{n-1})]) \\
&= f([\lambda_0, \ldots, \lambda_{n-1}, \lambda_n]).
\end{align*}

This fact also underlies the computation that $\zeta^{n+1} \circ {\delta_n} = \tilde{\delta^n} \circ \zeta^n$.  In other words, $\zeta^*$ preserves cocycles and coboundaries, so $\zeta^*$ induces an isomorphism ${H^*}(\Lambda, \A) \cong \tilde{H^*}(\Lambda, \A)$.
\end{proof}

\begin{rmk}
Applying Proposition \ref{prop:same-coh} to the case when $\Lambda$ is a higher-rank graph and $\A = A^\Lambda$ for an abelian group $A$ establishes that the cohomology of $\Lambda$, as defined in Definition \ref{def:module-cohlogy} above, agrees with the categorical cohomology  defined by Kumjian, Pask, and Sims in \cite{kps-twisted}.
\end{rmk}

\section{From $k$-graphs to groupoids and modules to sheaves}
\label{sec:kgraph-gpoid}
In this section we specialize from arbitrary small categories to the $k$-graphs {introduced by Kumjian and Pask in \cite{kp}} (see Definition \ref{def:kgraph} below), and their associated groupoids. 
We show that a module $\A$ over a $k$-graph $\Lambda$ gives rise to a sheaf $\shf{\A}$ over the associated groupoid $\G_\Lambda$; in fact, Theorem \ref{thm:mod-sheaf-functor} establishes that the map $\A \mapsto \shf{\A}$ is an exact functor. 

Sheaves over a groupoid $\G$, or $\G$-sheaves, play a central role in both the continuous cocycle cohomology and the sheaf cohomology of $\G$.  We discuss the continuous cocycle cohomology in this section, and postpone the discussion of sheaf cohomology to Section \ref{sec:equiv-sheaf}.  To be precise, we explain in Definitions \ref{def:set-sheaf} and \ref{def:Pn}
the construction of the $\G$-sheaves relevant to the continuous cocycle cohomology, and Proposition \ref{prop:cocycle-coh} establishes that the continuous cocycle cohomology of $\G$ can be computed from these sheaves.

For the reader who is not familiar with $k$-graphs, \'etale groupoids, and groupoid sheaves, we begin this section by defining these objects.  
Recall that $\N = \{0,1, 2, \ldots \}$ is a monoid under addition.  We regard  $\N^k$ as a category with one object generated by  $k$ commuting morphisms $e_1, \dots, e_k$.
\begin{defn}
\label{def:kgraph}(\cite{kp} Definition 1.1)
A \emph{higher-rank graph}, or {\em $k$-graph}, is a countable category $\Lambda$ equipped with a functor $d: \Lambda \to \N^k$ , such that $d$ satisfies the \emph{factorization property:} If a morphism $\lambda\in \Lambda$ satisfies $d(\lambda) = m+n$, then there exist unique $\mu, \nu \in \Lambda$ with $\lambda = \mu \nu$ and $d(\mu) = m, d(\nu) = n$.
\end{defn}


A fundamental example of a $k$-graph is the category $\Omega_k$, where $\text{Obj}\,\Omega_k = \N^k$ and $\text{Mor}\, \Omega_k = \{(m,n) \in \N^k \times \N^k: m \leq n\}$.  We have $r(m,n) = m, s(m,n) = n$, and composition is given by $(m,n)(n,p) = (m,p)$. The degree map $d: \Omega_k \to \N^k$ is given by $d(m,n) = n-m$.

For any $n \in \N^k$ {and any $k$-graph $\Lambda$}, we write $\Lambda^n = \{\lambda \in \Lambda: d(\lambda) = n\}$; 
we identify $\text{Obj} \, \Lambda = \Lambda^0$ and refer to elements of $\Lambda^0$ 
as  \emph{vertices}.  
If $v \in \Lambda^0$ we write 
\[
v\Lambda^n = \{\lambda\in \Lambda: r(\lambda) = v \text{ and } d(\lambda) = n\}.
\]

A $k$-graph $\Lambda$ is \emph{row-finite} if for all $m \in \N^k$ and all $v \in \Lambda^0$, we have $ |v\Lambda^m| < \infty$; we say $\Lambda$ has \emph{no sources} if $0 < |v\Lambda^m|$ for all $m, v$.  We will be exclusively concerned with row-finite $k$-graphs with no sources in this paper, since these are the $k$-graphs from which one can most easily build groupoids.\footnote{The row-finite and source-free requirement is slightly stronger than is strictly necessary in order to associate a groupoid to a $k$-graph; see \cite{farthing} for a more general groupoid approach to $k$-graphs.}

A \emph{groupoid} $\G$ is a small category with inverses.  
When discussing groupoids, as with $k$-graphs, we will use the arrows-only picture of category theory, identifying the objects $\G\z$ of $\G$ with their identity morphisms.  We will also often refer to the composition of morphisms in a groupoid $\G$ as ``multiplication,'' in line with our tendency to think of a groupoid as a generalization of a group, in which multiplication is defined only on a subset $\G\2$ of $\G \times \G$.

In this work, we will deal exclusively with 
\'etale groupoids $\G$; recall that $\G$ is \emph{\'etale} if {$\G$ has a locally compact Hausdorff topology with respect to which the multiplication and inverse operations are continuous, and }the range and source maps $r, s: \G \to \G\z$ are local homeomorphisms.

\begin{defn}(\cite{kp} Definition 2.1)
\label{def:groupoid}
For a  row-finite source-free $k$-graph $\Lambda$, the \emph{infinite path space} $\Lambda^\infty$ is the set of all $k$-graph morphisms $\Omega_k \to \Lambda$.  For $\lambda \in \Lambda$ write 

\[Z(\lambda) := \{x \in \Lambda^\infty: x(0,d(\lambda)) = \lambda\}\] and observe that the collection 
$\{ Z(\lambda) :  \lambda \in \Lambda\}$ forms a compact open basis for a topology on  $\Lambda^\infty$
For $p \in \N^k$, the map $\sigma^p: \Lambda^\infty \to \Lambda^\infty$ given by 
\[\sigma^p(x)(m,n) = x(m+p, n+p)\]
 is a local homeomorphism.
{Similarly, (cf.~Proposition 2.3 of \cite{kp})  for any $y \in \Lambda^\infty, \lambda \in \Lambda$ with $s(\lambda) = y(0)$, we define $\lambda y$ to be the unique path $x \in \Lambda^\infty$ such that $x(0,d(\lambda)) = \lambda$ and $\sigma^{d(\lambda)}(x) = y$. }
 
The groupoid $\G_\Lambda \subset  \Lambda^\infty \times \Z^k \times \Lambda^\infty$ associated to $\Lambda$ is 
\[
\G_\Lambda := \{(x, n, y): \exists \ j, \ell \in \N^k \text{ s.t. } j-\ell = n, \sigma^j(x) = \sigma^\ell(y)\}.
\]
We have $s(x,n,y) = (y, 0, y), r(x,n,y) = (x, 0,x)$, and composition given by $(x,n,y)(y, m,z)= (x, n+m, z)$.  Thus, we identify $\G_\Lambda\z $ with $ \Lambda^\infty$ via the embedding $\Lambda^\infty \ni x \mapsto (x, 0, x) \in \G_\Lambda\z$.

{Write $\Lambda *_s \Lambda = \{ (\lambda, \mu \in \Lambda \times \Lambda: s(\lambda) = s(\mu)\}.$} For each $(\lambda, \mu) \in \Lambda *_s \Lambda$, we define 
\[Z(\lambda, \mu) = \{(x,d(\lambda) - d(\mu), y): x(0, d(\lambda)) = \lambda, y (0, d(\mu)) = \mu\} \subseteq \G_\Lambda.\]
\end{defn}

{Proposition 2.8 of \cite{kp} establishes that} the sets $Z(\lambda,\mu)$ form a compact open basis for a locally compact Hausdorff topology on $\G_\Lambda$; in fact, this topology makes $\G_\Lambda$ into an \'etale groupoid.  

In what follows, we will discuss how to construct a sheaf $\shf{\A}$ over $\G_\Lambda$ from a $\Lambda$-module $\A$.  To this end, it will be helpful to use the \emph{espace \'etal\'e} picture of a sheaf.

\begin{defn}[\cite{hirzebruch} Definition I.2.1]
A \emph{sheaf}  of abelian groups over a topological space $X$ is a topological space $\mathscr{A}$ equipped with a local homeomorphism $\pi: \mathscr{A} \to X$ such that $\pi\inv(x)$ is an abelian group for all $x \in X$
and the fiberwise group operations  are continuous.
\end{defn}
Note that in particular there is a continuous zero section $X \to \mathscr{A}$, given by $x \mapsto 0_x$ 
where $0_x$ is the zero element in  $\mathscr{A}(x) := \pi\inv(x)$, the \emph{stalk} at $x$.

\begin{defn}[cf.~\cite{equiv-sheaf-coh} Definition 0.6]
\label{def:gpoid-sheaf}
Let $\G$ be an \'etale groupoid with unit space $\G\z$.  A  \emph{$\G$-sheaf} consists of a sheaf  of abelian groups, $\pi: \mathscr{A} \to \G\z$ over $\G\z$, equipped with  an isomorphism $\alpha_\gamma: \mathscr{A}(s(\gamma)) \to \mathscr{A}(r(\gamma))$ for each $\gamma \in \G$, such that:
\begin{itemize}
\item If $x \in \G\z$ then $\alpha_x = \text{id}$.
\item If $(\gamma_1, \gamma_2) \in \G\2$ then $\alpha_{\gamma_1} \circ \alpha_{\gamma_2} = \alpha_{\gamma_1 \gamma_2}$.
\item The map $\alpha: \G * \mathscr{A} \to \mathscr{A}$ given by $(\gamma, a) \mapsto \alpha_\gamma(a)$ is continuous. (We equip $\G * \mathscr{A} = \{(\gamma, a): s(\gamma) = \pi(a)\}$ with the subspace topology inherited from $\G \times \mathscr{A}$.)
\end{itemize}

Note that if $A$ is an abelian group one may form the constant sheaf $\shf{A}$ with fiber $\shf{A}(x) = A$ for all $x \in \G\z$ and with 
$\alpha_\gamma = \text{id}$ for all $\gamma \in \G$.

{ If $\mathscr{A}, \mathscr{B}$ are two $\G$-sheaves, we say a continuous map $f: \mathscr{A} \to \mathscr{B}$ is a \emph{morphism of $\G$-sheaves}, and write  $f \in \Hom_\G(\mathscr{A}, \mathscr{B})$, if:
\begin{itemize}
\item For any $x \in \G\z$, $f(a) \in \mathscr{B}(x)$ for all $a \in \mathscr{A}(x)$, and the induced map $\mathscr{A}(x) \to \mathscr{B}(x)$ is a homomorphism.
\item For any $(\gamma, a) \in \G * \mathscr{A}$, 
\[f(\alpha_\gamma (a)) = \beta_\gamma ( f(a))
\]
where $\beta$ is the action of $\G$ on $B$.
\end{itemize}}
\end{defn}
We often abuse notation and use the same symbol $\alpha$ to denote  the action of $\G$ on different $\G$-sheaves
and sometimes write simply $\gamma \cdot a$ for $\alpha_\gamma(a)$.
\begin{rmk}
\label{rmk:category-gsheaves}
The category of $\G$-sheaves $\S_\G$ is an abelian category (see \cite[Proposition 0.7]{equiv-sheaf-coh}).  
Subobjects and quotient objects are defined as in the category of sheaves with the additional requirement 
that the embeddings and quotient maps be  morphisms of $\G$-sheaves.  
The direct sum of two $\G$-sheaves $\mathscr{A}, \mathscr{B}$, denoted $\mathscr{A} \oplus \mathscr{B}$,
will also be regarded as a $\G$-sheaf when endowed with the diagonal action. \end{rmk}

\begin{rmk}
Since an \'etale groupoid $\G$ is a small category, we can also consider right (or left) $\G$-modules as in Definition \ref{def:module}.  Comparing Definition \ref{def:module} and Definition \ref{def:gpoid-sheaf}, we see that a $\G$-sheaf is in particular a left $\G$-module; but we require in addition that the action map  
$\alpha: \G * \mathscr{A} \to \mathscr{A}$ be continuous and that $\alpha_\gamma$ be an isomorphism for all $\gamma \in \G$.
\label{rmk:module-sheaf-compare}
\end{rmk}

Recall that a functor between abelian categories is said to be \emph{exact} if it preserves exact sequences. 

\begin{thm}
\label{thm:mod-sheaf-functor}
Let $\Lambda$ be a $k$-graph and let $\A$ be a $\Lambda$-module.  For each $x \in \Lambda^\infty$, define \[\shf{\A}(x) := \varinjlim 
(\A_{x(p)}, \A(x(p,q)) ).\]
where  $(p, q) \in \Omega_k$.\footnote{Note that $x(p,q)$ is the morphism in $\Lambda$ between the vertex $x(q) \in \Lambda^0 = \text{Obj}(\Lambda)$ and the vertex $x(p)$.  Since $\A$ is a contravariant functor, $\A(x(p,q)): \A_{x(p)} \to \A_{x(q)}$.}

There is a topology on $\shf{\A} := \bigsqcup_{x \in \Lambda^\infty} \shf{\A}(x)$ which makes $\shf{\A}$ into a $\G_\Lambda$-sheaf.  
Let $\eta \in \Hom_\Lambda(\A, \B)$;   
then for any $x \in \Lambda^\infty$,
\[
\varinjlim_{(p,q)\in \Omega_k}  \eta \circ \A(x(p,q)) = \varinjlim_{(p,q) \in \Omega_k} \B(x(p,q)) \circ \eta.
\]
We consequently obtain a well-defined element $\shf{\eta} \in  \Hom_{\G_\Lambda}(\shf{\A}, \shf{\B})$
such that  for any $[a] \in \shf{\mathcal{A}}(x)$
\[
\shf{\eta}([a]) := [\eta(a)]
\]
Moreover, the assignment $\A \mapsto \shf{\A}$ defines an  exact functor and the  induced map 
$\Hom_\Lambda(\A, \B) \to \Hom_{\G_\Lambda}(\shf{\A}, \shf{\B})$ is a homomorphism.
\end{thm}
\begin{proof}
Write $\varphi^x_p$ for the natural map $\A_{x(p)} \to \shf{\A}(x)$; the structure of the inductive limit implies that 
every element of $\shf{\A}(x)$ is of the form $\varphi^x_p(a)$ for some $p \in \N^k, a \in \A_{x(p)}$.
%
Equip $\shf{\A} = \bigsqcup_{x \in \Lambda^\infty} \shf{\A}(x)$ with the topology generated by the sets 
\begin{equation}
\label{eq:sheaf-open-sets}
\{S_{a, \lambda}\}_{\lambda \in \Lambda, a \in \A_{s(\lambda)}}, \text{ where } S_{a, \lambda} = \{\varphi^y_q(a): y(0, q) = \lambda\}.
\end{equation}
To see that these sets do indeed form a basis, suppose that $\beta \in S_{a, \lambda} \cap S_{b, \mu}$.  Then $\beta = \varphi^x_q(a) = \varphi^x_p(b) \in \shf{\A}(x)$ for some $x \in \Lambda^\infty$ such that $x(0,q) = \lambda$ and $x(0,p) = \mu$.  The definition of the inductive limit implies the existence of $m \geq p, q$  and $\lambda', \mu' \in \Lambda$ such that 
\[x(0,m) = \lambda \lambda' = \mu \mu' \text{ and } \A(\lambda')(a) = \A(\mu')(b) \in \A_{x(m)}.\]
Thus, $S_{\A(\lambda')(a), x(0,m)} = S_{\A(\mu')(b), x(0,m)}$; moreover, since $\varphi^x_p(b) = \beta = \varphi^x_q(a)$, 
\begin{align*}
\beta &= \varphi^x_{m}(\A(\lambda')(a)) = \varphi^x_{m}(\A(\mu')(b)) \\
\Rightarrow \beta &\in S_{\A(\lambda')(a), x(0,m)} = S_{\A(\mu')(b), x(0,m)} \subseteq S_{a, \lambda} \cap S_{b, \mu}.
\end{align*}
It follows that the sets $\{S_{a, \lambda}\}$ form a basis for a topology on $\shf{\A}$.
%

Note that the natural projection $\pi: \bigsqcup_{x\in \Lambda^\infty}\shf{\A}(x) \to \Lambda^\infty$ satisfies $\pi(S_{a,\lambda}) = Z(\lambda)$, so this topology does indeed make $\pi$ into a surjective local homeomorphism.  Moreover, the group operations are continuous in this topology: Suppose that $\varphi_q^y(a) - \varphi_p^y(b) \in S_{c, \lambda}$.  {We will show that there exist basic open sets $S_{m_a, \mu}, S_{m_b, \mu}$ such that 
\[\varphi^y_q(a) \in S_{m_a, \mu}, \ \varphi^y_p(b) \in S_{m_b, \mu}, \ \text{ and } \{ f-g: (f , g ) \in S_{m_a, \mu}*_\pi S_{m_b, \mu}\} \subseteq S_{c, \lambda}.\]}

Given $\varphi_q^y(a) - \varphi_p^y(b) \in S_{c, \lambda}$, there exists $ m' \geq p, q, d(\lambda)$ such that 
\[\varphi_q^y(a) - \varphi_p^y(b) = \varphi_{m'}^y(\A(y(q,m'))a) - \varphi_{m'}^y(\A(y(p,m'))b) \in S_{\A(y(\ell, m'))c, \lambda y(\ell, m')}.\]
 It follows that $\varphi_{m'}^y(\A(y(q,m'))a - \A(y(p,m'))b) = \varphi_{m'}^y(\A(y(\ell, m'))c).$
The structure of the direct limit then implies the existence of $m \geq m'$ such that 
\[\A(y(q,m))a - \A(y(p,m))b = \A(y(\ell, m))c.\]
Write $m_a := \A(y(q,m))a, \ m_b: = \A(y(p,m)) b$, and write $\mu = y(0,m)$.  Then $\varphi_q^y(a) \in S_{m_a, \mu}$ and $\varphi_p^y(b) \in S_{m_b, \mu}$. In fact, we claim that for any $f \in S_{m_a, \mu}, g \in S_{m_b, \mu}$ that both live in the same fiber $\shf{\A}(z)$, we have $f-g \in S_{\A(y(\ell, m))c, \mu} \subseteq  S_{c, \lambda}$.

To see this, choose $f, g$ as above, so $f= \varphi_m^z(m_a) ,g=\varphi_m^z(m_b)$. Consequently, 
\[f-g=\varphi_m^z(m_a-m_b)=\varphi_m^z(\A(y(\ell, m))c) \in S_{\A(y(\ell, m))c, \mu} \subseteq S_{c, \lambda},\]
so the fiber-wise group operations are continuous.

To make $\shf{\A}$ into a $\G_\Lambda$-sheaf, we equip it with the following  action $\alpha$ of $\G_\Lambda$.  If $(x, n, y) \in \G_\Lambda$ and $\shf{b} \in \shf{\A}(y)$,  write $\shf{b} = \varphi^y_q(b)$ for $b \in \A_{y(q)}$.  Without loss of generality, suppose that $q$ is large enough so that $y(q) = x(p)$ (where $n = p-q$); then, we define 
\begin{equation}
\label{eq:GLambda-action}
\alpha_{(x,n,y)}(\shf{b}) = (x,n,y) \cdot \shf{b} := \varphi^x_p(b).
\end{equation}

To see that $\alpha$ is continuous, suppose that $(x_i, n_i, y_i) \to (x,n,y) \in \G_\Lambda$ and that for all $i$, we have $\shf{b}_i \in \shf{\A}(y_i)$ with $\shf{b}_i \to \shf{b} \in \shf{\A}(y)$.  Suppose also that $(x,n,y) \cdot \shf{b} \in S_{a,\lambda}$; in other words, $x \in Z(\lambda)$ and there exists $q \in \N^k$ such that $y(q) = s(\lambda)$ and $\shf{b} = \varphi^y_q(a)$ for some $a\in \A_{y(q)} = \A_{s(\lambda)}$.

Writing $\mu = y(0,q)$, we have $(x,n,y) \in  Z(\lambda, \mu)$ and $\shf{b} \in S_{a,\mu}$.  Since $(x_i, n_i, y_i) \to (x,n,y)$ and $\shf{b}_i \to \shf{b}$, the fact that $\{Z(\lambda, \mu)\}_{\lambda, \mu}$ is a  basis for $\G_\Lambda$ implies that, for $i$ large enough, we also have $(x_i, n_i, y_i) \in Z(\lambda, \mu)$ and $\shf{b}_i \in S_{a,\mu}$, so $\shf{b}_i = \varphi^{y_i}_{q}(a)$.  It follows that
\[(x_i, n_i,y_i) \cdot \shf{b}_i = \varphi^{x_i}_{d(\lambda)}(a) \in S_{a,\lambda}\]
for large enough $i$.  Since $S_{a,\lambda}$ was an arbitrary neighborhood of $(x,n,y) \cdot \shf{b}$, the action $( (x,n,y), \shf{b}) \to \alpha_{(x,n,y)}(\shf{b})$ is continuous.

We now check the functoriality of the map $\A \mapsto \shf{\A}$.
{First, we check that a morphism $\eta: \A \to \B$ of $\Lambda$-modules  induces a continuous morphism of sheaves $\shf{\eta}:\shf{\A} \to \shf{\B}$ via the formula $\shf{\eta}[a] = [\eta(a)]$.  
In the notation used earlier in the proof, this formula becomes 
\[ \shf{\eta}(\varphi^y_q(a)) := \varphi^y_q(\eta(a)) \ \text{ for all } y \in \Lambda^\infty, \ a \in \A_{y(q)}.\]
Since $\eta$ is a natural transformation, we have $\eta \circ \A(y(p,q)) = \B(y(p,q))\circ \eta$ for all $y \in \Lambda^\infty$ and all $(p, q) \in \Omega_k$; this implies that $\shf{\eta} $ is well defined.
Moreover,  if $\shf{b}\in \shf{\eta}\inv(S_{b, \lambda})$, there exists $\mu \in \Lambda, y \in Z(\lambda \mu)$, $a \in \A_{s(\mu)}$ such that 
\[\shf{b} = \varphi^y_{d(\lambda \mu)}(a) \text{ and } \eta(a) = \B(\mu)(b).\]
In other words, $\shf{b} \in S_{a, \lambda \mu} $, and 
\begin{align*}
\shf{\eta}(S_{a,\lambda \mu})& = \{ \varphi^y_{d(\lambda \mu)}(\eta(a)): y \in Z(\lambda \mu)\} = \{\varphi^y_{d(\lambda \mu)}(\B(\mu)(b)): y \in Z(\lambda \mu)\}\\
&= \{ \varphi^y_{d(\lambda)}(b): y \in Z(\lambda \mu)\} \subseteq \shf{\B}\\
&\subseteq S_{b, \lambda}.
\end{align*}
 That is, every $\shf{b} \in \shf{\eta}\inv(S_{b, \lambda})$ has an open neighborhood $S_{a, \lambda \mu}$ which is contained in $\shf{\eta}\inv(S_{b,\lambda})$.  Consequently, $\shf{\eta}$ is continuous, as claimed.}

It follows immediately from the definitions  that $\shf{id} = id$ and that if $\eta, \rho$ are morphisms of $\Lambda$-modules, 
then $\shf{\eta \circ \rho} = \shf{\eta} \circ \shf{\rho}$.  So the assignment $\A \mapsto \shf{\A}$ is a functor.  
It is routine to check that the map $\Hom_\Lambda(\A, \B) \to \Hom_{\G_\Lambda}(\shf{\A}, \shf{\B})$ is a homomorphism.

Let $0 \to \A \to \B \to \C \to 0$ be a short exact sequence of $\Lambda$-modules.  
Then since the inductive limit of short exact sequences is a short exact sequence we have
$0 \to \shf{\A}(x) \to \shf{\B}(x) \to \shf{\C}(x) \to 0$  is exact for all $x \in \Lambda^\infty$.
Thus, $0 \to \shf{\A} \to \shf{\B} \to \shf{\C} \to 0$ is exact in $\S_{\G_\Lambda}$.
It follows that the functor $\A \mapsto \shf{\A}$ is exact, as desired.
\end{proof}

We now give the details of the above construction in the important special case of the $\G_\Lambda$-sheaves associated to the projective $\Lambda$-modules $P^n$.
\begin{example}
\label{ex:shf-Fn}

For any $n \in \N$, we have 
\[\shf{P^n}(x) \cong \Z\{((x, m, y), (\lambda_1, \ldots, \lambda_n) ) \in \G_\Lambda \times \Lambda^{*n}: s(\lambda_n) = r(y)\}.\]
For each $x \in \Lambda^\infty$, the map  $\shf{d_n}(x): \shf{P^n}(x) \to \shf{P^{n-1}}(x)$ is given on the generators by
\begin{align*}
\shf{d_n}(x)[(x,m,y),( \lambda_1, \ldots, \lambda_n)] &= [ (x,m,y),(\lambda_2, \ldots, \lambda_n) ]\\
& \quad +\sum_{j=1}^{n-1} (-1)^j [ (x,m, y),(\lambda_1, \ldots, \lambda_j \lambda_{j+1}, \ldots, \lambda_n)] \\
& \quad + (-1)^n[ (x,m-d(\lambda_n), \lambda_n y), ( \lambda_1, \ldots, \lambda_{n-1})] ;\\
\shf{d_1}(x)[(x,m,y),\lambda] &= [x,m,y] - [x, m-d(\lambda),\lambda y]; \\
\shf{d_0}(x)[x,m,y] &= 1_x.
\end{align*}
\begin{proof}
Fix $x \in \Lambda^\infty$, and write  $\shf{G^n}(x)$ for the free abelian group described above; that is, 
$\shf{G^n}(x) := \Z\{((x, m, y), (\lambda_1, \ldots, \lambda_n) ) \in \G_\Lambda \times \Lambda^{*n}: s(\lambda_n) = r(y)\}$. As above, we will use square brackets to indicate a generator of $\shf{G^n}(x)$.

 For each $p \in \N^k$, we have a map $\phi^x_p: P^n_{x(p)} \to \shf{G^n}(x)$ given on the generators by  
\[\phi^x_p([\lambda_0, \ldots, \lambda_n]) = [(x,p- d(\lambda_n), \lambda_n \sigma^p(x)), (\lambda_0, \ldots, \lambda_{n-1})].\]
Note that $\phi^x_q \circ P^n(x(p,q)) = \phi^x_p$, so the universal property of the direct limit implies that the maps $\{\phi^x_p\}_{p \in \N^k}$ induce a homomorphism $\phi: \shf{P^n}(x) \to \shf{G^n}(x)$.  

Moreover, every generator $[(x,m,y),(\lambda_1, \ldots, \lambda_n)]$ of $\shf{G^n}(x) $ is in the image of $\phi^x_p$ for some $p$: namely, 
\[ [(x, p-q, y),(\lambda_1, \ldots, \lambda_n)] = \phi^x_p([\lambda_1, \ldots, \lambda_n, y(0, q)]).  \]
This surjectivity implies that, if $\{\rho_{x(p)}\}_{p \in \N^k}: P^n_{x(p)} \to A$ is any family of group homomorphisms that is compatible with the maps $P^n(x(p,q)): P^n_{x(p)} \to P^n_{x(q)}$, then  we can define $R: \shf{G^n}(x) \to A$ by 
\[R\left( \phi^x_p ([\lambda_0, \lambda_1, \ldots, \lambda_n)]\right) = \rho_{x(p)}([\lambda_0, \ldots, \lambda_n]).\]
The universal property of the inductive limit now implies that $\shf{G^n}(x) = \shf{P^n}(x)$, as claimed.  The formula for $\shf{d_n}$ follows from the definition and universal properties of $\shf{P^n}$.
\end{proof}


Finally, note that for each $ \gamma =(x,m, y) \in \G_\Lambda$ and each $n \in \N$, the action $\alpha_\gamma: \shf{P^n}(y) \to \shf{P^n}(x)$ is given on generators as follows.
If $b = [\lambda_0, \ldots, \lambda_n]$ is a generator of $P^n_{y(p-m)}$ for some {$p  \in \N^k$ with $p \geq m$},  we have \[
\phi^y_{p-m}(b) = [(y,p-m-d(\lambda_n), \lambda_n \sigma^{p-m}(y)),(\lambda_0, \ldots, \lambda_{n-1})].
\]  
In this case, with $\shf{b} := \phi^y_{p-m}(b)$,
\begin{align*}\alpha_\gamma(\shf{b}) = (x, m, y) \cdot \shf{b} =& (x,m,y) \cdot \phi^y_{p-m}([\lambda_0, \ldots, \lambda_n]) = \phi^x_p([\lambda_0, \ldots, \lambda_n])  \\
&= [(x, p-d(\lambda_n), \lambda_n \sigma^p(x)), (\lambda_0, \ldots, \lambda_{n-1})].
\end{align*}
Note that since $(x,m,y) \in \G_\Lambda$ and $p \geq m$, we have $p - (p-m) = m$ and hence $\sigma^p(x) = \sigma^{p-m}(y)$.
\end{example}

 {In general, given $(x, m, y)$ and $(y, \ell, z) \in \G_\Lambda$, we can write $ z = \lambda \sigma^{p-m}(y)$ for some $p \in \N^k$ with $p \geq m$ and some $\lambda := z(0, d(\lambda)) \in \Lambda$.  By construction, we have 
$p-m-d(\lambda) = \ell$, and consequently $m + \ell = p-d(\lambda)$. 

If $(\lambda_0, \ldots, \lambda_{n-1}) \in \Lambda^{*n}$ satisfies $s(\lambda_{n-1}) = r(\lambda)$, we have 
 \[ \phi^y_{p-m}([\lambda_0, \ldots, \lambda_{n-1}, \lambda]) = [(y, \ell, z), (\lambda_0, \ldots, \lambda_{n-1})].\]  Thus, we obtain a general formula for the action of $\G_\Lambda$ on $\shf{P^n}$:
 \begin{equation}
 (x,m,y) \cdot [(y, \ell, z), (\lambda_0, \ldots, \lambda_{n-1})] = [(x, m+\ell, z), (\lambda_0, \ldots, \lambda_{n-1})]
 \label{eq:Fn-action}
 \end{equation}
 
 }

{
\subsection{$\G$-sheaves from sheaves of sets}

It is not always obvious whether a bundle $\mathscr{A}$ of abelian groups over $\G\z$ can be topologized in a way that makes it into a $\G$-sheaf.  If $\mathscr{A}$ arises from a $\G$-sheaf of sets, however, this is always possible.  As a matter of fact, all of the $\G$-sheaves we discuss in this paper have this format, so we pause to detail this construction.

\begin{defn}
\label{def:set-sheaf}
Let $\G$ be an \'etale groupoid. A topological space $Y$, equipped with an onto local homeomorphism $\pi: Y \to \G\z$, is a \emph{(left) $\G$-sheaf of sets}  if, for each $\gamma \in \G$, there is a bijection $\alpha_\gamma: \pi^{-1}(s(\gamma)) \to \pi^{-1}(r(\gamma))$ such that 
\begin{itemize}
\item If $x \in \G\z$ then $\alpha_x = id$.
\item If $(\gamma_1, \gamma_2) \in \G\2$ then $\alpha_{\gamma_1} \circ \alpha_{\gamma_2} = \alpha_{\gamma_1 \circ \gamma_2}$.
\item The map $\alpha: \G * Y \to Y$ given by $(\gamma, y) \mapsto \alpha_\gamma(y)$ is continuous.
\end{itemize}
\end{defn} }
 If $Y$ is a $\G$-sheaf of sets over $\G^{(0)}$,
then one may form a $\G$-sheaf of abelian groups $\Z[Y]$ over $\G^{(0)}$ 
such that the stalk at $x \in \G^{(0)}$ is  the free abelian group generated by $Y(x) := \pi^{-1}(x)$. 
When it is useful to distinguish between an element $y$ in $Y(x)$ from the corresponding generator of $\Z[Y](x),$ we use $[y]$ for the latter.  

The topology on $\Z[Y]$  can be described as follows.  Since $Y$ is a $\G$-sheaf of sets, for each point $e \in Y(x)$, there is an open neighborhood $\omega_e$ of $e$ and an open neighborhood $U_e$ of $x$ such that $\omega_e$ is homeomorphic to $U_e$. For each $z \in U_e$, write $e_z$ for the unique element of $Y(z) \cap \omega_e$. 

Given an arbitrary point $a$ in $\Z[Y](x)$, write $a$ as a finite sum, 
\[a = \sum_{e \in Y(x)} a_e [e],\]
 where $a_e \in \Z$. 
  Since this sum is finite, there is a finite set $Y^a(x) \subset Y(x)$ such that $a_e = 0$ if $e \not\in Y^a(x)$, and so 
  the set 
  \[\bigcap_{e \in Y^a(x)}  U_e \subseteq \G\z\]
   is open.  
  For any open $U \subseteq \cap_{e \in Y^a(x)} U_e$ such that $x \in U$, we define 
\[\mathcal{O}_{a, U} = \left\{\sum_{{e \in Y^a(x)} } a_e [e_z]: z \in U\right\}. \]
The sets $\{\mathcal{O}_{a,U}\}_{a, U}$ form a basis for the topology on $\Z[Y]$.

The action of $\G$  on  $\Z[Y]$ is determined by its action on $Y$: 
on generators, $\gamma \cdot [e] =  [\gamma  e]$ for $e \in Y({s(\gamma)})$. We then extend the action $\Z$-linearly.

Checking that $\{\mathcal{O}_{a,U}\}$ is indeed a basis, and that the group operations and the $\G$-action are  continuous, is analogous to the proof of Theorem \ref{thm:mod-sheaf-functor}.

{
\begin{example}
In Example \ref{ex:shf-Fn}, we showed  that the fiber over $x \in \Lambda^\infty$ of the $\G_\Lambda$-sheaf $\shf{P^n}$ is of the form $\Z[Y_n](x)$, where 
\[
Y_n := 
\{ ((x, m, y), ( \lambda_1, \ldots, \lambda_n)) \in \G_\Lambda \times \Lambda^{*n} : 
r(y) = s(\lambda_n) \}.
\]
In fact, we can topologize $Y_n$ so that $\Z[Y_n] \cong \shf{P^n}$ as $\G_\Lambda$-sheaves.  We thus obtain a second picture of the sheaves $\shf{P^n}$, complementing  their initial definition in Example \ref{ex:shf-Fn} as an inductive limit of $\Lambda$-modules.

First, we make $Y_n$ into a  sheaf of sets over $\G_\Lambda\z$ via the sheaf map $\pi: Y_n \to \G_\Lambda\z = \Lambda^\infty$ given by 
\[
\pi((x, m, y), ( \lambda_1, \ldots, \lambda_n)) = x.
\]
The topology on $Y_n$ is generated by the sets 
\[  Z(\lambda, \mu) \times \{ (\lambda_1, \ldots, \lambda_n)\},\]
with $ (\lambda, \mu) \in \Lambda *_s \Lambda, (\lambda_1, \ldots, \lambda_n) \in \Lambda^{*n},$ and $ s(\lambda_n) = r(\mu), $
where $Z(\lambda, \mu)$ is as in Definition \ref{def:groupoid}.  The $\G_\Lambda$-action on $Y_n$ is given by
\[ (z, \ell, x) \cdot ( (x, m,y),(\lambda_1,\ldots, \lambda_n)) := ((z, \ell+m, y), (\lambda_1, \ldots, \lambda_n)).\]

 Under this definition, not only do $\Z[Y_n]$ and $\shf{P^n}$ agree as sets, but the action of $\G_\Lambda$ on the generators is given by the same formula in both cases. We now check that the topology on the $\G_\Lambda$-sheaf $\Z[Y_n]$ agrees with the topology on $\shf{P^n}$ defined in Theorem \ref{thm:mod-sheaf-functor}. 
Let $F$ be a finite set and write 
\[ B := \sum_{i\in F} b_i [(x, m_i, y_i), (\lambda_1^i, \ldots, \lambda_n^i)] \in \Z[Y_n](x)\]
for integers $b_i$.
Choose a sufficiently small  open neighborhood $Z(\lambda)$ of $x$ such that, for each $i$,
we have  $y_i = \mu_i \sigma^{d(\lambda)}(x)$ for some $\mu_i$ with $m_i = d(\lambda) - d(\mu_i)$.  Setting $A = \sum_{i\in F} b_i [\lambda_1^i, \ldots, \lambda^i_n, \mu_i] \in P^n_{s(\lambda)},$ we have 
\begin{align*}
\mathcal{O}_{B, Z(\lambda)} &= \left\{ \sum_{i \in F} b_i [(\lambda z, m_i, \mu_i z), (\lambda_1^i, \ldots, \lambda_n^i)]: r(z) = s(\lambda) \right\} = S_{A, \lambda}.
\end{align*}
Similarly, if $C = \sum_j c_j [\lambda_0^j, \ldots, \lambda_n^j] \in P^n_{s(\lambda)}$ is arbitrary and $x \in Z(\lambda)$, let 
\[ D:= \sum_j c_j [(x, d(\lambda)- d(\lambda_n^j), \lambda_n^j \sigma^{d(\lambda)}(x)), (\lambda_0, \ldots, \lambda_{n-1})] \in \Z[Y_n](x).\]
One checks immediately that $\mathcal{O}_{D, Z(\lambda)} = S_{C, \lambda}$.  In other words, the identity map $\Z[Y_n] \to \shf{P^n}$ is a homeomorphism.
\label{ex:Fn-set-sheaf}
\end{example}

For another example, 
 we return to the (right) projective modules of Definition \ref{def:proj-res}.  Here, however, our focus will be on  $\G$-sheaves where the groupoid $\G$ acts on the left.  In the definition that follows, we also modify the notation slightly in order to mesh more easily with the standard notation for continuous groupoid $n$-cocycles (see Proposition \ref{prop:cocycle-coh} below).   

Write $\G^{(n)}$ for the set of composable $n$-tuples of elements of $\G$, for each $n \geq 0$.  (This set was denoted $\G^{*n}$ in Definition \ref{def:proj-res}.) The topology on $\G^{(n)} \subseteq \G \times \cdots \times \G $ is  the  subspace topology; consequently, the fact that the range and source maps in an \'etale groupoid are   local homeomorphisms ensures that the map $\pi: \G^{(n)} \to \G\z$ given by 
\[ \pi(\gamma_1, \ldots, \gamma_n) = r(\gamma_1)\]
is a surjective local homeomorphism.  With the left action of $\G$ given by left multiplication,  
\[
 \gamma \cdot (\gamma_1, \ldots, \gamma_n) = (\gamma\gamma_1, \ldots, \gamma_n),
 \]
$\G^{(n)}$ satisfies the conditions of Definition \ref{def:set-sheaf}.  Consequently, the remarks following Definition \ref{def:set-sheaf} tell us how to topologize $\Z [\G^{(n)}]$ in order to make it a $\G$-sheaf.  As a module, we denoted $\Z [\G^{(n)}]$ by $P^{n-1}$ in Definition \ref{def:proj-res}; when we wish to regard $\Z [\G^{(n)}]$ as a $\G$-sheaf, we will use the notation $\mathscr{P}_{n-1}$.

\begin{defn}
\label{def:Pn}
For each $n \ge 0$ we define the left $\G$-sheaf $\mathscr{P}_n := \Z[\G^{(n+1)}]$.  
Fix $n \ge 0$ and $x \in \G\z$; then
\[
(\mathscr{P}_n)(x) = \Z\{ (\gamma_0,\gamma_1, \ldots, \gamma_n) \in \G^{(n+1)}: r(\gamma_0) = x\}.
\]
For $n \ge 1$ define $\partial_n: \mathscr{P}_n \to \mathscr{P}_{n-1}$ on the generators by 
\[
\partial_n[\gamma_0,\gamma_1, \ldots, \gamma_n] = (-1)^{n} [\gamma_0, \ldots, \gamma_{n-1}] 
 + \sum_{i=1}^n (-1)^{i-1} [\gamma_0,\gamma_1, \ldots, \gamma_{i-1} \gamma_{i}, \ldots, \gamma_n]
 \]
If $n=0$, we define $\partial_0: \mathscr{P}_0 \to \shf{\Z}$ on the generators by 
\[\partial_0[\gamma] = [1]_{r(\gamma)}.\]
 \end{defn}

\begin{prop}\label{prop:script-p}
{For any \'etale groupoid $\G$,} the sequence $\{ \mathscr{P}_n, \partial_n\}_{n\in \N}$ constitutes a resolution of {the constant $\G$-sheaf} $\shf{\Z}$. 
\end{prop}
\begin{proof}
 It is straightforward to check that $\partial_n$ is an equivariant sheaf map
and $\partial_n \partial_{n+1} = 0$ for each $n \ge 0$; the continuity of $\partial_n$ is a consequence of the continuity of the multiplication in $\G$.

 Indeed, the maps $\partial_n$ are simply a translation into the setting of left $\G$-sheaves of the maps $d_n$ of Definition \ref{def:proj-res}.  Consequently, just as in the proof of Proposition \ref{prop:proj-res}, we have a contracting homotopy $\{s_n: \mathscr{P}_n \to \mathscr{P}_{n+1}\}_n$ given on generators by 
\[ s_n([\gamma_0, \ldots, \gamma_n]) = [r(\gamma_0), \gamma_0, \ldots, \gamma_n]; \qquad   s_{-1}([1]_x) = [x].\]
It follows that $\ker \partial_n = \text{Im}\, \partial_{n+1}$ for all $n$.  The foregoing proves the following.
\end{proof}
{It  can be shown that $\{\mathscr{P}_n, \partial_n\}_n$ is in fact a relative projective resolution of $\shf{\Z}$, 
but we shall not need this here so we omit the details.}

\subsection{Continuous cocycle cohomology}
One can use the resolution $\{\mathscr{P}_n, \partial_n\}_n$ to compute the continuous cocycle cohomology of a groupoid $\G$ with coefficients in a $\G$-sheaf $\mathscr{A}$; this is the content of Proposition \ref{prop:cocycle-coh} below.

For the definition of the continuous cocycle cohomology $H^n_c(\G, \mathscr{A})$, we follow \cite[Definition 1.11 ff]{renault} except that we do not require the cocycles be normalized.  As in the case of $k$-graph cocycles, this does not change the cohomology groups, since every cocycle is cohomologous to a normalized cocycle.

\begin{defn}
\label{def:cocycle-coh}{\cite[Definitions 1.11-1.12]{renault}}
Let $\mathscr{A}$ be a
 $\G$-sheaf.  
The  set of \emph{continuous groupoid $n$-cochains} with values in $\mathscr{A}$ 
 is defined to be
\[
C^n_c(\G, \mathscr{A}) := 
\{ f : \G^{(n)} \to \mathscr{A} : f \text{ continuous, }
f(\gamma_1, \ldots, \gamma_n) \in \mathscr{A}_{r(\gamma_1)} \}.
\]
We regard $C^n(\G, \mathscr{A})$ as an abelian group under pointwise addition.
\end{defn}
Alternatively, we may regard 
$C^n_c(\G, \mathscr{A})$ as the group of continuous sections of the pull-back sheaf
$r^*( \mathscr{A}) \to \G^{(n)}$, where $r(\gamma_1, \ldots, \gamma_n) = r(\gamma_1)$.

We form a cohomology complex by defining the boundary maps
$\delta^n_c: C^n_c(\G, \mathscr{A}) \to C^{n+1}_c(\G, \mathscr{A})$ as follows.
Let $f \in C^n_c(\G, \mathscr{A})$ and $(\gamma_0, \gamma_1, \ldots, \gamma_n) \in \G^{(n+1)}$.
If $n \ge 1$ set
\begin{multline*}
\delta^n_c f(\gamma_0, \gamma_1, \ldots, \gamma_n) := \gamma_0 \cdot f(\gamma_1, \ldots, \gamma_n)
+ \sum_{i=1}^n (-1)^i f(\gamma_0, \dots, \gamma_{i-1}\gamma_i, \dots,\gamma_n) \\
+ (-1)^{n+1}f(\gamma_0, \gamma_1, \ldots, \gamma_{n-1});
\end{multline*}
 if $n = 0$ set $\delta^0_c f(\gamma_0) := \gamma_0 \cdot f(s(\gamma_0)) - f(r(\gamma_0))$.
A straightforward computation shows that $\delta^{n+1}_c \circ \delta^n_c = 0$.

Define 
the group of $n$-cocycles 
$Z^n_c(\G, \mathscr{A}) := \ker \delta^n_c$ and the group of $n$-coboundaries
$B^n_c(\G, \mathscr{A}) := 
 {\text{Im}\, \delta^{n-1}_c}$ for $n \ge 1$; set
$B^0_c(\G, \mathscr{A}) := 0$.  
The \emph{$n$th continuous cocycle cohomology group} of $\G$ with coefficients in $\mathscr{A}$ is then defined to be 
\[
H^n_c(\G, \mathscr{A}) = Z^n_c(\G, \mathscr{A})/B^n_c(\G, \mathscr{A}) = \frac{\ker \delta^n_c}{\text{Im}\, \delta^{n-1}_c}.
\]

{For any   resolution $\{\mathscr{Q}_n, \vartheta_n\}_{n\in \N}$ of $\shf{\Z}$, the cohomology groups of the associated complex $\Hom_\G(\mathscr{Q}_*, \mathscr{A})$ often contain useful information about the cohomology of $\G$.  In addition to Proposition \ref{prop:cocycle-coh} below, which shows that the continuous cocycle cohomology can be computed from the resolution $\{ \mathscr{P}_n, \partial_n\}_n$ of Definition \ref{def:Pn}, we will see this principle at work in Proposition \ref{prop:exact-to-sheaf} in the next section.}

The  boundary maps  $\delta_n: \Hom_\G(\mathscr{Q}_n, \mathscr{A}) \to \Hom_\G(\mathscr{Q}_{n+1}, \mathscr{A})$ of the  complex $\Hom_\G(\mathscr{Q}_*, \mathscr{A})$ are given by $
\delta_n f =  f \circ \vartheta_{n+1}.$
We define $Z^n_{\mathscr Q}(\G,\mathscr{A}) := \ker \delta_n$ for $n \ge 0$,  
$B^n_{\mathscr{Q}}(\G,\mathscr{A}) := \text{Im}\,\delta_{n-1}$ for $n \ge 1$ and 
$B^0_{\mathscr{Q}}(\G,\mathscr{A}) := 0$.
Then the  cohomology groups {of the complex}  are 
\begin{equation}
H^n_{\mathscr{Q}}(\G, \mathscr{A}) := H^n(\Hom_{\G}(\mathscr{Q}_*, \mathscr{A})) = \frac{Z^n_{\mathscr{Q}}(\G, \mathscr{A})}{B^n_{\mathscr{Q}}(\G, \mathscr{A})}.
\label{eq:exact-coh}
\end{equation}



\begin{prop}
\label{prop:cocycle-coh}
{Let $\G$ be an \'etale groupoid and let $\mathscr{A}$ be a $\G$-sheaf.}
For all $n \geq 0$, 
there is an isomorphism 
$\xi^n : \Hom_\G(\mathscr{P}_n, \mathscr{A}) \to C^n_c(\G, \mathscr{A})$ determined by
\[
(\xi^nf)(\gamma_1, \ldots, \gamma_n) = f([r(\gamma_1), \gamma_1, \ldots, \gamma_n]),
\]
which is compatible with the boundary maps and induces an isomorphism
\[
H^n_{\mathscr P}(\G, \mathscr{A}) \cong H^n_c(\G, \mathscr{A}) 
\qquad \text{for every } n \ge 0.
\]
The inverse is induced by $\eta^n : C^n_c(\G, \mathscr{A}) \to  \Hom_\G(\mathscr{P}_n, \mathscr{A})$ determined by
\[
(\eta^n f )([\gamma_0, \gamma_1, \ldots, \gamma_n ]) = \gamma_0 \cdot f(\gamma_1, \ldots, \gamma_n).
\]

\end{prop}
\begin{proof}
We first check that for each sheaf morphism $f \in \Hom_\G(\mathscr{P}_n, \mathscr{A})$,  $\xi^n f$ is  continuous.  
If $\{(\gamma_1^i, \ldots, \gamma_n^i)\}_i \to (\gamma_1, \ldots, \gamma_n)$ in $\G^{(n)}$, the fact that $r$ is a local homeomorphism implies that $\{ (r(\gamma_1^i), \gamma_1^i, \ldots, \gamma_n^i)\}_i \to (r(\gamma_1), \gamma_1, \ldots, \gamma_n)$ in $\G^{(n+1)}$ and hence in $\mathscr{P}_n = \Z[\G^{(n+1)}]$.   
The fact that any morphism $f$ of $\G$-sheaves is continuous then implies that $\xi^n f$ is continuous for all $f \in \Hom_\G(\mathscr{P}_n, \mathscr{A})$.

Moreover, for any $(\gamma_0, \ldots, \gamma_n) \in \G^{(n+1)}$ and any $f \in \Hom_\G(\mathscr{P}_n, \mathscr{A})$,
\begin{equation}
\label{eq:pn-equivar}
\gamma_0 \cdot f[r(\gamma_1), \gamma_1, \ldots, \gamma_n] = f[\gamma_0, \gamma_1, \ldots, \gamma_n].
\end{equation}
A routine computation, exploiting this fact, shows that 
\[\delta^n_c(\xi^n f) = \xi^{n+1} ( f \circ \partial_{n+1});\] 
in other words, that $\xi^n$ takes cocycles to cocycles and coboundaries to coboundaries.  Consequently, it induces a homomorphism $ H^n_{\mathscr P}(\G, \mathscr{A})) \to H^n_c(\G, \mathscr{A})$.

To see that $\xi^n$ is an isomorphism, 
recall that  $\eta^n : C^n_c(\G, \mathscr{A}) \to \Hom_\G(\mathscr{P}_n, \mathscr{A})$ 
is given by
\[
\eta^n f \left(\sum_{i\in F} a_i [\gamma_0^i, \gamma_1^i, \ldots, \gamma_n^i ]\right) = \sum_{i\in F} a_i\left(  \gamma_0^i \cdot f(\gamma_1^i, \ldots, \gamma_n^i)\right).
\]
One checks easily that $\eta^nf$ is $\G$-equivariant for all $f \in C^n_c(\G, \mathscr{A})$, {and the continuity of $\eta^n f$ follows from the continuity of the $\G$-action on $\mathscr{A}$, the continuity of $f$, and the fact that if 
\[\lim_J \sum_{i \in F_J} a_{i,J} [\gamma_0^{i, J}, \ldots, \gamma_n^{i,J}] = \sum_{i \in F} a_i [\gamma_0^i, \ldots, \gamma_n^i] \in \mathscr{P}_n,\]
 then we must have $F_J = F$ eventually and, at that point, $a_{i, J} = a_i$ for all $i \in F$.}
 Moreover,  for each $n \ge 0$, \eqref{eq:pn-equivar} implies that  $\xi^n\eta^n = \text{id}_{C^n_c(\G, \mathscr{A})}$ and
$\eta^n\xi^n = \text{id}_{\Hom(\mathscr{P}_n, \mathscr{A})}$.
\end{proof}

\begin{rmk}
An analogue of Proposition \ref{prop:cocycle-coh} was proved in \cite[5.1]{equiv-sheaf-coh}, using a relative injective resolution of $\mathscr{A}$.
\end{rmk}


\section{A commuting diagram of cohomologies} 
\label{sec:equiv-sheaf}

In this section, we use  sheaf cohomology to show how the $k$-graph  cohomology  $H^n(\Lambda, \A)$ and the continuous cocycle cohomology $H^n_c(\G_\Lambda, \shf{\A})$ relate.  More precisely, Theorem \ref{thm:mod-sheaf-functor} and
Proposition \ref{prop:cocycle-coh} above, and  Proposition \ref{prop:exact-to-sheaf} below, allow us to identify homomorphisms 
\begin{align*}
H^n(\Lambda, \A) \to H^n_{\shf{P}}(\G, \shf{\A}) &\xrightarrow{\rho^n_{\shf{P}}} H^n(\G_\Lambda, \shf{\A}) \\
\quad \text{ and } \quad H^n_c(\G_\Lambda, \shf{\A}) \xrightarrow{\cong} H^n_{\mathscr{P}}(\G_\Lambda, \shf{A})  &\xrightarrow{\rho^n_{\mathscr P}} H^n(\G_\Lambda, \shf{\A}),
\end{align*}
where $H^n(\G_\Lambda, \shf{\A})$ denotes the sheaf cohomology of $\G_\Lambda$ with coefficients in $\shf{\A}$. 
We also construct a map $\psi_n^*: H^n_c(\G_\Lambda, \shf{\A}) \to H^n_{\shf P} (\G, \shf{\A}))$ and 
show that the map from the continuous cocycle cohomology to the sheaf cohomology factors through $\psi_n^*$.

Thus, the main result of this section is the following:
\begin{thm}
\label{thm:diagram-commutes}
For any row-finite higher-rank graph $\Lambda$ with no sources, and any $\Lambda$-module $\A$, we have a commuting diagram (for $n \ge 2$):
\begin{equation}\label{eq:maps}
\begin{tikzpicture}
    \node (02) at (0,2) {$H^n(\Lambda, \A)$};
    \node (42) at (4,2) {$H^n_{\shf P}(\G_\Lambda, \shf{\A}))$};
    \node (00) at (0,0) {$H^n_c(\G_\Lambda,  \shf{\A}))$};
    \node (40) at (4,0) {$H^n_{\mathscr{P}}(\G_\Lambda, \shf{\A}))$};
    \node (82) at (8,2) {$H^n(\G_\Lambda,  \shf{\A}))$};
    \draw[->] (40) -- node[right] {${\psi_n^*}$} (42);
    \draw[->] (00) -- node[above] {${{\eta^n}}$} node[below] {${\scriptstyle{\cong}}$} (40);
    \draw[->] (02) -- (42);
    \draw[->] (02) -- node[left] {\cite{kps-twisted}}  node[right]{($n \leq 2$)} (00);
    \draw[->] (40) -- node[below] {${\rho^n_{\mathscr{P}}}$} (82);
    \draw[->] (42)-- node[above] {${\rho^n_{\shf{P}}}$} (82);    
\end{tikzpicture}
\end{equation}
\end{thm}

We prove Theorem \ref{thm:diagram-commutes} via a series of propositions.  First, we describe the sheaf cohomology groups $H^n(\G, \mathscr{A})$ of $\G$ with coefficients in a $\G$-sheaf $\mathscr{A}$, and construct the maps $\rho^n_{\mathscr{P}}, \rho^n_{\shf{P}}$.  Both maps are special cases of a more general construction, detailed in Proposition \ref{prop:exact-to-sheaf} below.  The bottom row of the diagram \eqref{eq:maps} was established in Proposition \ref{prop:cocycle-coh}.  The vertical arrow $\psi^*_n$ and the commutativity of the right-hand triangle of \eqref{eq:maps} are established in Proposition \ref{pr:psi-isom}.  
Propositions  \ref{pr:psi-sigma-same} and \ref{prop:kps-same-zero} connect our work with that of Kumjian, Pask, and Sims in \cite{kps-twisted}, and complete the proof of Theorem \ref{thm:diagram-commutes}, by showing that the homomorphisms $H^n(\Lambda, A) \to H^n_c(\G_\Lambda, \shf{A})$ established in Section 6 of \cite{kps-twisted} for $n \leq 2$ make the left-hand square of \eqref{eq:maps} commute. 

We conclude the paper with a few remarks on the question of when the maps in the diagram \eqref{eq:maps} are isomorphisms.  Kumjian, Pask, and Sims  established in \cite{kps-twisted} that the left-most vertical arrow need not be surjective  but conjectured that it is injective; however, Example \ref{ex:not-isom} shows this conjecture to be false.  Finally, Remark \ref{rmk:ample-gpoid-sheaf} shows that the diagonal arrow is an isomorphism for $n = 2$.

The equivariant sheaf cohomology of an \'etale groupoid $\G$ was introduced in \cite{equiv-sheaf-coh}, inspired by Grothendieck's work in \cite{grothendieck}. See also \cite{haefliger-diff-coh} for an alternative approach.


A $\G$-sheaf $\mathscr{I}$ is said to be \emph{injective} if every injective map $f \in \Hom_\G(\mathscr{I}, \mathscr{A})$ 
has a left inverse (i.e.\ a map $g \in \Hom_\G(\mathscr{A}, \mathscr{I})$ such that $gf = id_\mathscr{I}$).
\cite[Corollary 1.6]{equiv-sheaf-coh} establishes that the category of $\G$-sheaves $\S_\G$  has enough injectives in 
the sense that every $\G$-sheaf may be embedded into an injective  $\G$-sheaf, enabling the following definition.


\begin{defn}
\label{def:sheaf-coh}
Let $\G$ be an \'etale groupoid.
A \emph{cohomology theory for $\G$} is a sequence of covariant functors {$\{H^n(\G, \cdot )\}_{n\in \N}$} from the category $\S_\G$ of $\G$-sheaves to the category of abelian groups 
which satisfy the following conditions:
\begin{itemize}
\item[i.]  $H^0(\G, \cdot ) \cong \Hom_\G(\shf{\Z}, \cdot)$.
\item[ii.]  $H^n(\G, \mathscr{I}) = 0$ if $\mathscr{I}$ is injective and $n > 0$.
\item[iii.]  For each short exact sequence of  $\G$-sheaves, $0 \to  \mathscr{A} \to  \mathscr{B} \to  \mathscr{C} \to 0$, there are
natural connecting maps $\delta^n : H^n(\G, \mathscr{C}) \to  H^{n+1}(\G, \mathscr{A})$ such that the following sequence is exact:
\begin{gather*}
0 \to H^0(\G, \mathscr{A}) \to H^0(\G, \mathscr{B}) \to H^0(\G, \mathscr{C}) \xrightarrow{\delta^0} H^1(\G, \mathscr{A}) \to \cdots \\
\xrightarrow{\delta^{n-1}} H^n(\G, \mathscr{A}) \to H^n(\G, \mathscr{B}) \to H^n(\G, \mathscr{C}) \xrightarrow{\delta^n}  H^{n+1}(\G, \mathscr{A}) \to \cdots
\end{gather*}
\end{itemize}
\end{defn}
The cohomology functors defined this way are unique up to natural equivalence, and will be denoted the \emph{(equivariant) sheaf cohomology} of $\G$.

{
\begin{rmk}\label{rmk:ext}
Since $\S_\G$ is an abelian category that has enough injectives, 
the bifunctors $\Ext^n_\G(\cdot, \cdot)$ 
which classify $n$-fold exact sequences of $\G$-sheaves are defined. 
Moreover, Section 1 of \cite{equiv-sheaf-coh} indicates that $H^n(\G, \cdot)$ and $\Ext_\G^n(\shf{\Z}, \cdot)$ are naturally equivalent. 
Thus, for a $\G$-sheaf $\mathscr{A}$ and $n \geq 1$ we may identify elements of $H^n(\G, \mathscr{A})$ as 
equivalence classes of $n$-fold exact sequences of $\G$-sheaves
\[
0 \to \mathscr{A} \to \mathscr{R}_{n-1} \to \cdots \to \mathscr{R}_0 \to \shf{\Z} \to 0.
\]
(When $n = 0$, we recall that $\Ext^0_\G(\shf{\Z}, \cdot) = \Hom_\G(\shf{\Z} ,\cdot) = H^0(\G, \cdot)$, so the correspondence between extensions and sheaf cohomology also occurs when $n =0$.)
The equivalence relation is generated by morphisms between $n$-fold exact  sequences which identify the ends, $\mathscr{A}$ and $\shf{\Z}$,  as in the following  commutative diagram.  
\begin{equation*}
\begin{tikzpicture}[xscale=0.7,yscale=0.6, >=stealth]
    \node (11) at (0,2) {$0$};
    \node (21) at (2,2) {$\mathscr{A}$};
    \node (31) at (4,2) {$\mathscr{R}_{n-1}$};
    \node (41) at (6,2) {$\cdots$};
    \node (51) at (8,2) {$ \mathscr{R}_{0}$};
    \node (61) at (10,2) {$\shf{\Z}$};
    \node (71) at (12,2) {$0$};
    \node (10) at (0,0) {$0$};
    \node (20) at (2,0) {$\mathscr{A}$};
    \node (30) at (4,0) {$\mathscr{S}_{n-1}$};
    \node (40) at (6,0) {$\cdots$};
    \node (50) at (8,0) {$ \mathscr{S}_{0}$};
    \node (60) at (10,0) {$\shf{\Z}$};
    \node (70) at (12,0) {$0$};
    \draw[->] (11) -- (21);
    \draw[->] (21) -- (31);
    \draw[->] (31) -- (41);
    \draw[->] (41) -- (51);
    \draw[->] (51) -- (61);
    \draw[->] (61) -- (71);
    \draw[->] (10) -- (20);
    \draw[->] (20) -- (30);
    \draw[->] (30) -- (40);
    \draw[->] (40) -- (50);
    \draw[->] (50) -- (60);
    \draw[->] (60) -- (70);
    \draw (21) edge[double distance=2pt]  (20);
    \draw[->] (31) -- (30);
    \draw[->] (51) -- (50);
    \draw (61) edge[double distance=2pt] (60);
\end{tikzpicture}
\end{equation*}
See page 215 of \cite{equiv-sheaf-coh} and \cite[Proposition III.5.2]{maclane}. 
We shall use this characterization of  $H^n(\G, \mathscr{A})$ in the sequel.
\end{rmk}


The following Proposition, which 
defines the two right-most arrows of \eqref{eq:maps}, relies on the terminology introduced in Remark \ref{rmk:category-gsheaves} and Equation \eqref{eq:exact-coh}.

\begin{prop}
\label{prop:exact-to-sheaf}
Let $\{\mathscr{Q}_n, \vartheta_n\}_n$ be a resolution of the constant $\G$-sheaf $\shf{\Z}$.  
Let $c \in  Z^n_{\mathscr{Q}}(\G, \mathscr{A})$.  
For $n \geq 1$ we have a commuting diagram
\begin{equation*}
\begin{tikzpicture}[xscale=0.85,yscale=0.6, >=stealth]
    \node (01) at (-2,2) {$\cdots$};
    \node (11) at (0,2) {$\mathscr{Q}_{n+1}$};
    \node (21) at (2,2) {$\mathscr{Q}_n$};
    \node (31) at (4,2) {$\mathscr{Q}_{n-1}$};
    \node (41) at (6,2) {$\mathscr{Q}_{n-2}$};
    \node (51) at (8,2) {$\cdots$};
    \node (61) at (10,2) {$\mathscr{Q}_{0}$};
    \node (71) at (11.5,2) {$\shf{\Z}$};
    \node (81) at (12.7,2) {$0$};
    \node (10) at (0,0) {$0$};
    \node (20) at (2,0) {$\mathscr{A}$};
    \node (30) at (4,0) {$\mathscr{R}_{c}$};
    \node (40) at (6,0) {$\mathscr{Q}_{n-2}$};
    \node (50) at (8,0) {$ \cdots$};
    \node (60) at (10,0) {$\mathscr{Q}_{0}$};
    \node (70) at (11.5,0) {$\shf{\Z}$};
    \node (80) at (12.7,0) {$0$};
    \draw[->] (01) -- (11);
    \draw[->] (11) -- node[above] {${\scriptstyle \vartheta_{n+1}}$}(21);
    \draw[->] (21) -- node[above] {${\scriptstyle \vartheta_n}$} (31);
    \draw[->] (31) -- node[above] {${\scriptstyle \vartheta_{n-1}}$}(41);
    \draw[->] (41) -- node[above] {${\scriptstyle \vartheta_{n-2}}$}(51);
    \draw[->] (51) -- node[above] {${\scriptstyle \vartheta_{1}}$}(61);
    \draw[->] (61) -- node[above] {${\scriptstyle \vartheta_{0}}$}(71);
    \draw[->] (71) -- (81);
    \draw[->] (10) -- (20);
    \draw[->] (20) -- node[above] {${\scriptstyle \iota}$}(30);
    \draw[->] (30) -- node[above] {${\scriptstyle \tilde{\vartheta}_{n-1}}$}(40);
    \draw[->] (40) -- node[above] {${\scriptstyle \vartheta_{n-2}}$}(50);
    \draw[->] (50) -- node[above] {${\scriptstyle \vartheta_{1}}$}(60);
    \draw[->] (60) -- node[above] {${\scriptstyle \vartheta_{0}}$}(70);
    \draw[->] (70) -- (80);
   \draw[->] (11) -- (10);
    \draw[->] (21) --  node[left] {${\scriptstyle c}$}(20);
    \draw[->] (31) -- (30); 
    \draw (41) edge[double distance=2pt] (40);
    \draw (61) edge[double distance=2pt] (60);
    \draw (71) edge[double distance=2pt] (70);
\end{tikzpicture}
\end{equation*}
with exact rows.  
Moreover, the class of this $n$-fold exact sequence in $\Ext_\G^n(\shf{\Z} , \mathscr{A})$ only depends on $[c] \in H^n_{\mathscr Q}(\G, \mathscr{A})$.
For $n = 0$ there is a unique $\tilde{c} \in \Hom_\G(\shf{\Z}, \mathscr{A})$ such that 
$c = \tilde{c} \circ \vartheta_{0}$ for $c \in  Z^0_{\mathscr{Q}}(\G, \mathscr{A})$.

Furthermore, for all $n \in \N$ we obtain a well-defined homomorphism 
\[
\rho^n_\mathscr{Q}: H^n_{\mathscr{Q}}(\G, \mathscr{A}) \to H^n(\G, \mathscr{A})
\] 
which for $n \ge 1$ is given by taking $[c] \in  H^n_{\mathscr{Q}}(\G, \mathscr{A})$ 
to the class of the $n$-fold exact sequence 
\begin{equation} 
0 \to \mathscr{A} \xrightarrow{\iota} \mathscr{R}_c \xrightarrow{\widetilde{\vartheta}_{n-1}} 
\mathscr{Q}_{n-2} \xrightarrow{\vartheta_{n-2}} \cdots \xrightarrow{\vartheta_{1}} 
\mathscr{Q}_0 \xrightarrow{\vartheta_{0}} \shf{\Z} \to 0,
\label{eq:curly-SES}
\end{equation}
and for $n=0$, we have $\rho^0_\mathscr{Q}([c]) = \tilde{c}$  
for $c \in  Z^0_{\mathscr{Q}}(\G, \mathscr{A})$.
\end{prop}
\begin{proof}
We begin by defining the objects in the sequence \eqref{eq:curly-SES}.
Given an $n$-cocycle $c \in Z^n_{\mathscr{Q}}(\G, \mathscr{A}),$  $\mathscr{R}_c$ is the pushout of the diagram 
\[ 
\begin{tikzpicture}[yscale=0.8]
   \node (02) at (0,2) {$\mathscr{Q}_n$};
    \node (00) at (0,0) {$\mathscr{A}$};
    \node (22) at (2,2) {$\mathscr{Q}_{n-1}$};
    \draw[->] (02) -- node[left] {${\scriptstyle c}$} (00);
    \draw[->] (02) -- node[above] {${\scriptstyle \vartheta_n}$} (22);
\end{tikzpicture}
\]
That is, $\mathscr{R}_c= (\mathscr{A} \oplus \mathscr{Q}_{n-1} ) / \{ (c(q), -\vartheta_n(q)): q \in \mathscr{Q}_n\}$ (see  Remark \ref{rmk:category-gsheaves}).

We define $\widetilde{\vartheta}_{n-1}: \mathscr{R}_c \to \mathscr{Q}_{n-2}$ by 
$\widetilde{\vartheta}_{n-1}[a, p] = \vartheta_{n-1}(p)$
and $\iota: \mathscr{A} \to \mathscr{R}_c$ by $\iota(a) = [a, 0]$.
The exactness of the complex $\{\mathscr{Q}_n, \vartheta_n\}_n$ implies that $\widetilde{\vartheta}_{n-1}$ is well defined; showing that the  sequence of $\G$-sheaves 
\eqref{eq:curly-SES}
is exact relies on the exactness of the complex and the fact that $c \circ \vartheta_{n+1}=0$ for any $n$-cocycle $c$.  
We will write $\rho^n_{\mathscr{Q}}: Z_{\mathscr{Q}}^n(\G, \mathscr{A}) \to H^n(\G, \mathscr{A})$ for the map taking a cocycle $c$ to the exact sequence \eqref{eq:curly-SES}.

In fact, $\rho^n_{\mathscr{Q}}$ induces a map 
\[ \rho^n_{\mathscr{Q}}: H^n_{\mathscr{Q}}(\G, \mathscr{A}) \to H^n(\G, \mathscr{A}).\]
This follows from the observation that if $c-d \in B^n_{\mathscr{Q}}(\G, \mathscr{A})$, so that $c-d = g \circ \partial_n$ for some $g \in \Hom_\G(\mathscr{Q}_{n-1}, \mathscr{A})$, the map 
\[ \mathscr{R}_c \to \mathscr{R}_d \quad \text{ given by }\quad [a, p] \mapsto [a + g(p), p]\]
is a $\G$-equivariant sheaf homomorphism which (combined with the identity maps on the other sheaves in the exact sequence $\rho^n_{\mathscr Q}(c)$) induces a morphism $\rho^n_\mathscr{Q}(c) \to \rho^n_\mathscr{Q}(d)$.  In other words, 
\[[c]=[d] \in H^n_{\mathscr{P}}(\G, \mathscr{A}) \Rightarrow [\rho^n_{\mathscr Q}(c)] = [\rho^n_{\mathscr{Q}}(d)] \in  H^n(\G, \mathscr{A}) ,\]
so $\rho^n_{\mathscr Q}$ gives a well-defined homomorphism of cohomology groups. 

For the case $n=0$, recall that $Z^0_\mathscr{Q}(\G, \mathscr{A}) = H^0_\mathscr{Q}(\G, \mathscr{A})$.
Let $c \in Z^0_\mathscr{Q}(\G, \mathscr{A})$.   
Then, since $c \circ \vartheta_1 =0$, the exactness of the sequence 
\[ 
\cdots \to \mathscr{Q}_1 \stackrel{\vartheta_1}{\to} \mathscr{Q}_0 \stackrel{\vartheta_0}{\to} \shf{\Z} \to 0 
\]
at $\mathscr{Q}_0$ implies that there is a unique $\tilde{c} \in \Hom_\G(\shf{\Z}, \mathscr{A})$ such that 
$c = \tilde{c} \circ \vartheta_{0}$. 
It is routine to show that the resulting map 
 \[
 \rho^0_\mathscr{Q}: H^0_\mathscr{Q}(\G, \mathscr{A}) \to \Hom_\G(\shf{\Z}, \mathscr{A}) = H^0(\G, \mathscr{A})
 \]
is a homomorphism.
\end{proof}



{To complete the right-hand triangle of the diagram \eqref{eq:maps}, we now describe the maps $\psi_n: \shf{P^n} \to \mathscr{P}_n$ which induce the vertical maps $\psi_n^*$ in the diagram.}
 

Recall from Example \ref{ex:Fn-set-sheaf} that $\shf{P^n} = \Z[Y_n]$.  Given $((x, m, y), ( \lambda_1, \ldots, \lambda_n)) \in Y_n$, let
\begin{align*}
\gamma_0 &:= (x, m, y) \\
\gamma_1 &:= (y, -d(\lambda_n), \lambda_n y) \\
\gamma_2 &:= (\lambda_n y, -d(\lambda_{n-1}), \lambda_{n-1}\lambda_n y) \\
         &\quad \vdots \\
\gamma_n &:= (\lambda_2\cdots\lambda_n y, -d(\lambda_1), \lambda_1\cdots\lambda_n y),        
\end{align*}

The map 
$((x, m, y), ( \lambda_1, \ldots, \lambda_n)) \mapsto (\gamma_0, \gamma_1, \dots , \gamma_n)$ 
allows us to define $\psi_n: \shf{P^n} \to \mathscr{P}_n$ by setting 
\begin{equation}\label{eq:psi}
\psi_n([(x, m, y), ( \lambda_1, \ldots, \lambda_n)]) 
= (-1)^{\lceil{n/2}\rceil}[\gamma_0, \gamma_1, \dots , \gamma_n]
\end{equation}
for $n \geq 0$, and extending $\Z$-linearly.  
We define $\psi_{-1}: \shf{\Z} \to \shf{\Z}$ 
to be the identity map.

\begin{prop}
\label{pr:psi-isom}
{Let $\Lambda$ be a row-finite, source-free $k$-graph and let $\G_\Lambda$ be its associated groupoid. For any $n \in \N$,
the maps $\psi_n: \shf{P^n}\to \mathscr{P}_n$ are continuous, equivariant sheaf morphisms and 
\[ 
\psi_{n-1} \circ \shf{d_n} = \partial_n \circ \psi_n.
\]  
Furthermore, for any $\G_\Lambda$-sheaf $\mathscr{A}$,}  the induced map $\psi_n^*: H^n_{\mathscr{P}}(\G_\Lambda, \mathscr{A}) \to H^n_{\shf{P}}(\G_{\Lambda}, \mathscr{A})$ satisfies
\[ 
\rho^n_{\shf{P}} \circ \psi_n^* =  \rho^n_{\mathscr{P}}.
\]
It follows that the right-hand 
triangle of \eqref{eq:maps} commutes.
\end{prop}

\begin{proof}

To see that $\psi_n$ is continuous for $n \in \N$, fix $a \in \mathscr{P}_n(x)$ such that $a \in \text{Im}\, \psi_n$.  Writing $a = \sum_{i\in F} a_i [(\gamma_0^i, \ldots, \gamma_n^i)] \in \text{Im}\, \psi_n$, we have $r(\gamma_0^i) = x$ for all $i$.  There exists $\lambda \in \Lambda$ such that $x \in Z(\lambda)$ and that $Z(\lambda)$ is homeomorphic (under the projection map $\pi(\gamma_0, \ldots, \gamma_n) = r(\gamma_0)$) to an open neighborhood of each $(\gamma_0^i, \ldots, \gamma_n^i)$. In other words, for each $i \in F$, we have $z_i \in \Lambda^\infty, \ (\lambda_1^i, \cdots, \lambda_n^i, \mu^i) \in \Lambda^{*(n+1)}$ such that
\begin{align*}
\gamma_0^i &= (\lambda z_i, d(\lambda) - d(\mu^i), \mu^i z_i) \\
\gamma_1^i &= (\mu^i z_i, -d(\lambda^i_n), \lambda_n \mu^i z_i) \\
& \vdots \\
\gamma_n^i &= (\lambda^i_2 \cdots \lambda^i_n \mu^i z_i, -d(\lambda^i_1), \lambda_1^i \cdots \lambda_n^i \mu^i z_i).
\end{align*} 
If $z \in Z(\lambda)$, for $0 \leq j \leq n$, define $\gamma_j^i(z) $ by replacing $z_i$ in the formula above for $\gamma_j^i$ with $\sigma^{d(\lambda)}(z)$.  Then 
\[\mathcal{O}_{a, Z(\lambda)} = \left\{ \sum_{i \in F} a_i [(\gamma_0^i(z), \cdots, \gamma_n^i(z) )]: z \in Z(\lambda)\right\}.\]
 Using the maps $\phi^x_p: P^n_{x(p)} \to \shf{P^n}(x)$ from Example \ref{ex:shf-Fn}, and the topology on $\shf{P^n}$ given in the proof of Theorem \ref{thm:mod-sheaf-functor}, we see that 
\begin{align*}
\psi_n^{-1}(\mathcal{O}_{a, Z(\lambda)}) &= \left\{ (-1)^{\lceil n/2 \rceil}\sum_{i\in F} a_i [(z, d(\lambda)-d(\mu^i), \mu^i \sigma^{d(\lambda)}(z)), (\lambda_1^i, \ldots, \lambda_n^i)] : z \in Z(\lambda) \right\}\\
&= \left\{ \phi^z_{d(\lambda)} \left(\sum_{i \in F} (-1)^{\lceil n/2 \rceil} a_i [(\lambda_1^i, \ldots, \lambda_n^i, \mu^i)]\right): z \in Z(\lambda) \right\}\\
&= S_{\tilde{a}, \lambda}
\end{align*}
is a basic open set in $\shf{P^n}$, where $\tilde{a} = (-1)^{\lceil n/2 \rceil}\sum_{i \in F} a_i [(\lambda_1^i, \ldots, \lambda_n^i, \mu^i)] \in P^n$.  In other words, $\psi_n$ is continuous; it is straightforward to check that $\psi_n$ is also equivariant, and consequently defines an element in $\Hom_{\G_\Lambda}( \shf{P^n}, \mathscr{P}_n)$.
We now proceed to show that the $\psi_n$  intertwine the boundary maps. 

For $((x, m, y), ( \lambda_1, \ldots, \lambda_n)) \in Y_n$ we have
\begin{align*}
 (-1)^{\lceil{(n-1)/2}\rceil}&\psi_{n-1}(\shf{d_n}[(x, m, y), (\lambda_1, \dots, \lambda_n)]) \\
      &=   (-1)^{\roof{(n-1)/2}}\psi_{n-1}([(x, m, y), (\lambda_2, \dots, \lambda_n)])\\
     & \quad  +  (-1)^{\roof{(n-1)/2}}\sum_{i=2}^n (-1)^{i-1} \psi_{n-1}([(x, m, y), (\lambda_1, \dots, \lambda_{i-1}\lambda_i, \dots, \lambda_n)]) \\
     & \quad +  (-1)^{\roof{(n-1)/2}}(-1)^n\psi_{n-1}([(x, m-d(\lambda_n), \lambda_n y), (\lambda_1, \dots, \lambda_{n-1})]) \\
       & = [ \gamma_0, \gamma_1, \dots, \gamma_{n-1}] - [ \gamma_0, \gamma_1, \dots, \gamma_{n-1}\gamma_n] + \\
      & \qquad \cdots  + (-1)^{n-1} [ \gamma_0, \gamma_1\gamma_2, \dots, \gamma_n]  
       + (-1)^n [ \gamma_0\gamma_1, \gamma_2, \dots, \gamma_n]  \\
       & =  [ \gamma_0, \gamma_1, \dots, \gamma_{n-1}] +
        \sum_{i=1}^n (-1)^{n-i+1}  [ \gamma_0, \dots, \gamma_{i-1}\gamma_i, \dots, \gamma_n] \\
       & = (-1)^{\roof{n/2}}(-1)^n\partial_n\psi_{n}[(x, m, y), (\lambda_1, \dots, \lambda_n)].
\end{align*}
Hence, since $\roof{(n-1)/2} + \roof{n/2} = n$, we have
\[ 
\psi_{n-1} \circ \shf{d_n} =  (-1)^{\roof{(n-1)/2} + \roof{n/2} + n}\partial_n \circ \psi_n 
=\partial_n \circ \psi_n.  
\]

To establish the commutativity of the right-hand triangle of \eqref{eq:maps}, we first note  that {the calculation above establishes that}
\[
\psi_n^*(c) := c \circ \psi_n
\] 
is an $n$-cocycle in $Z^n_{\shf{P}}(\G_\Lambda, \mathscr{A})$ whenever $c \in Z^n_{\mathscr P}(\G_\Lambda, \mathscr{A})$.  
We must prove that 
\begin{equation}
\label{eq:right-triangle-commutativity}
[\rho^n_{\mathscr P}(c)]  = [\rho^{n}_{\shf{P}}\psi_n^*(c)] .
\end{equation}
{This is immediate when $n = 0$, since $\mathscr{P}_0 = \Z[\G] = \shf{P^0}$, and $\psi_0$ is the identity map.}

{When $n \geq 1$,} Equation \eqref{eq:right-triangle-commutativity} follows from the observation that the map $\chi: \mathscr{R}_{\psi_n^*(c)} \to \mathscr{R}_c $ given by 
\[
\chi[a, p] = [a, \psi_{n-1}(p)]
\]
 is well-defined and induces, together with $\{\psi_i\}_{i=0}^{n-2}$, a map between the exact sequences 
associated to $c$ and $\psi_n^*(c)$ as in Proposition \ref{prop:exact-to-sheaf}.    That is, the following is a commutative diagram with exact rows:
\begin{equation*}
\begin{tikzpicture}[xscale=0.8,yscale=0.7, >=stealth]
    \node (11) at (0.5,2) {$0$};
    \node (21) at (2,2) {$\mathscr{A}$};
    \node (31) at (4,2) {$\mathscr{R}_{\psi_n^*(c)}$};
    \node (41) at (6,2) {$\shf{P_{n-2}}$};
    \node (51) at (8,2) {$\cdots$};
    \node (61) at (10,2) {$\shf{P_{0}}$};
    \node (71) at (12,2) {$\shf{\Z}$};
    \node (81) at (13.5,2) {$0$};
    \node (10) at (0.5,0) {$0$};
    \node (20) at (2,0) {$\mathscr{A}$};
    \node (30) at (4,0) {$\mathscr{R}_{c}$};
    \node (40) at (6,0) {$\mathscr{P}_{n-2}$};
    \node (50) at (8,0) {$ \cdots$};
    \node (60) at (10,0) {$\mathscr{P}_{0}$};
    \node (70) at (12,0) {$\shf{\Z}$};
    \node (80) at (13.5,0) {$0$};
    \draw[->] (11) -- (21);
    \draw[->] (21) -- (31);
    \draw[->] (31) -- (41);
    \draw[->] (41) -- node[above] {${\scriptstyle \shf{d_{n-2}}}$}(51);
    \draw[->] (51) -- node[above] {${\scriptstyle \shf{d_{1}}}$}(61);
    \draw[->] (61) -- node[above] {${\scriptstyle \shf{d_{0}}}$}(71);
    \draw[->] (71) -- (81);
    \draw[->] (10) -- (20);
    \draw[->] (20) -- (30);
    \draw[->] (30) -- (40);
    \draw[->] (40) -- node[below] {${\scriptstyle \partial_{n-2}}$}(50);
    \draw[->] (50) -- node[below] {${\scriptstyle \partial_{1}}$}(60);
    \draw[->] (60) -- node[below] {${\scriptstyle \partial_{0}}$} (70);
    \draw[->] (70) -- (80);
    \draw (21) edge[double distance=2pt] (20);
    \draw[->] (31) -- node[left] {${\scriptstyle \chi}$}(30); 
    \draw[->] (41) -- node[left] {${\scriptstyle \psi_{n-2}}$} (40);
        \draw[->] (61) -- node[left] {${\scriptstyle \psi_{0}}$} (60);
    \draw (71) edge[double distance=2pt] (70);
\end{tikzpicture}
\end{equation*}

 This proves \eqref{eq:right-triangle-commutativity}; thus $\rho^n_{\shf{P}} \circ \psi_n^* =  \rho^n_{\mathscr{P}}$ and so 
  the right-hand half of \eqref{eq:maps} commutes.
\end{proof}

We now complete the proof of Theorem \ref{thm:diagram-commutes} by establishing
the commutativity of the left-hand square in the diagram \eqref{eq:maps} in the case which was considered by Kumjian, Pask, and Sims in \cite{kps-twisted}.  To that end, let $A$ be an abelian group. In \cite{kps-twisted}, the authors describe  homomorphisms ${H}^n(\Lambda, A) \to H^n_c(\G, \shf{A})$ for $n \leq 2$ (see Lemma 6.3, Remark 6.9, and the paragraph preceding Remark 6.9 of that paper).  We show in Propositions   \ref{pr:psi-sigma-same} and  \ref{prop:kps-same-zero} below that these homomorphisms make the left-hand square of  \eqref{eq:maps} commutative.

{To mesh better with the notation from \cite{kps-twisted}, in what remains of this section we will use the description of $H^n(\Lambda, A)$ in terms of categorical cocycles, as in Definition \ref{defn:categorical}, rather than the $\Lambda$-module perspective.}

\begin{prop}
\label{pr:psi-sigma-same}
Let $A$ be an abelian group, and let $c$ be a categorical $A$-valued 2-cocycle on a row-finite $k$-graph $\Lambda$ with no sources. Let $c \mapsto \sigma_c$ be the map ${H}^2(\Lambda, A) \to H^2_c(\G_\Lambda, \shf{A})$ of \cite[Lemma 6.3]{kps-twisted}.  We have 
\[ 
\psi_2^*(\eta^2(\sigma_c)) = \shf{c}.
\]
Hence, when $n = 2$, the left-hand square of the diagram \eqref{eq:maps} commutes.
\end{prop}
\begin{proof}
We begin by observing that, since the action of $\Lambda$ on $A$ is trivial,
\begin{align*} 
\shf{c}[(x, m, y), (\lambda_1, \lambda_2)]  = (\zeta^2)^{-1} c(\lambda_1, \lambda_2, r(y)) = c(\lambda_1, \lambda_2).
\end{align*}
where $\zeta^*$ denotes the isomorphism between the categorical cohomology and the $\Lambda$-module cohomology given in Proposition \ref{prop:same-coh}.

Recall that $\eta^n$ is the inverse to the isomorphism $\xi^n: H^n_{\mathscr{P}}(\G, \shf{A}) \to H^n_c(\G, \shf{A})$ of Proposition \ref{prop:cocycle-coh}.  
Then using formula (\ref{eq:psi})  we have
\begin{align*}
\psi_2^*(\eta^2(\sigma_c))& [(x,m,y), (\lambda_1, \lambda_2)] \\
&= (-1)^{\roof{2/2}} (x, m, y) \cdot \sigma_c((y, -d(\lambda_2), \lambda_2 y), (\lambda_2 y, -d(\lambda_1), \lambda_1 \lambda_2 y)) \\
&= - \sigma_c((y, -d(\lambda_2), \lambda_2 y), (\lambda_2 y, -d(\lambda_1), \lambda_1 \lambda_2 y)) ,
\end{align*}
since we posited a trivial action of $\Lambda$ (and hence of $\G_\Lambda$) on $A$.

In order to evaluate $\sigma_c ((y, -d(\lambda_2), \lambda_2 y), (\lambda_2 y, -d(\lambda_1), \lambda_1 \lambda_2 y))$, we must find a collection $\mathcal{P}$ of cylinder sets $Z(\mu, \nu)$ which forms a partition of the groupoid $\G_\Lambda$ as in Lemma 6.6 of \cite{kps-twisted}.

Reviewing the proof of the above-cited Lemma 6.6, we see that we can use the techniques presented in that proof to construct a partition $\mathcal{P}$ of $\G_\Lambda$ such that for any $\lambda \in \Lambda$, we have $Z(\lambda, s(\lambda)) \in \mathcal{P}$ and $Z(s(\lambda), \lambda) \in \mathcal{P}$.

Thus, since $g := (y, -d(\lambda_2), \lambda_2 y)$ and $h:= (\lambda_2 y, -d(\lambda_1), \lambda_1 \lambda_2 y)$, as well as their product $gh= (y, -d(\lambda_1 \lambda_2), \lambda_1 \lambda_2 y)$ are in cylinder sets of the form $Z(s(\lambda), \lambda)$, we have (in the notation of Lemma 6.3 of \cite{kps-twisted}) 
\[\mu_g = \mu_{gh} = \alpha  =  \gamma =r(y) ; \quad \nu_g = \beta = \lambda_2; \quad \mu_h = r(\lambda_2); \quad \nu_h =  \lambda_1; \quad \nu_{gh} = \lambda_1 \lambda_2.\]
  Therefore, assuming that the 2-cocycle $c$ is normalized (which assumption we can always make, up to cohomology), we have
\begin{align*}
\psi_2^*(\eta^2&\,(\sigma_c))[(x,m, y),    (\lambda_1, \lambda_2)] \\
& = -\sigma_c ((y, -d(\lambda_2), \lambda_2 y), (\lambda_2 y, -d(\lambda_1), \lambda_1 \lambda_2 y))  \\
&= - (c(\mu_g, \alpha) - c(\nu_g, \alpha) + c(\mu_h, \beta) - c(\nu_h, \beta) - c(\mu_{gh}, \gamma) + c(\mu_{gh}, \gamma)) \\ 
 &= -(c(r(y), r(y)) - c(r(y), \lambda_2) + c(r(\lambda_2), \lambda_2) - c(\lambda_1, \lambda_2) \\
& \qquad - c(r(y), r(y)) + c(\lambda_1\lambda_2, r(y))) \\
&= c(\lambda_1, \lambda_2) \\
&  = \shf{c}[(x, m, y), (\lambda_1, \lambda_2)]. \qedhere
\end{align*}
\end{proof}

\begin{prop}
For an abelian group $A$ and $n \leq 1$, the homomorphisms $H^n(\Lambda, A) \to H^n(\G_\Lambda, \shf{A})$ identified in the paragraphs following Corollary 6.8 of \cite{kps-twisted} make the left-hand square of \eqref{eq:maps} commute.
\label{prop:kps-same-zero}
\end{prop}
\begin{proof}
Recall from \cite{kps-twisted} that when $n=1,$ the map $H^1(\Lambda, A) \to H^1(\G_\Lambda, \shf{A})$ (which we will denote $c \mapsto \sigma^1_c$  in analogy with Proposition \ref{pr:psi-sigma-same}) is given by 
\[ 
\sigma_c^1(y, \ell - j, z) =  c(y(0,\ell)) - c(z(0, j)),
\]
and when $n=0$ we have $\sigma^0_f[y]= f(r(y))$.
Thus, 
\begin{align*}
\psi_0^*(\eta^0(\sigma^0_f))[(x, m, y)] &= (x, m, y) \cdot f(r(y)) = f(r(y)), \ \text{ and }\\
\psi_1^*(\eta^1(\sigma^1_c))[(x, m, y), \lambda] &= (-1)(x, m, y) \cdot \left( c(r(y)) - c(\lambda) \right) \\
&= c(\lambda)
 \end{align*}
whenever $c$ is a normalized cocycle.
Moreover, for a 1-cocycle $c \in Z^1(\Lambda, A),$ chasing through the formulas from Theorem \ref{thm:mod-sheaf-functor} and Proposition \ref{prop:same-coh} reveals that \[
\shf{c}[(x, m, y), \lambda] = c(\lambda)\]
as well, and for a 0-cocycle $f$ we also have $\shf{f}[(x, m,y)] = f(r(y)) = {\psi}_0^*(\eta^0(\sigma^0_f))$.
\end{proof}

We remark that the homomorphisms $H^n(\Lambda, A) \to H^n(\G_\Lambda, \shf{A})$ established by Kumjian, Pask, and Sims will not, in general, be  isomorphisms.  
Remark 6.9 of \cite{kps-twisted} presents an example of a 1-graph $B_2$ {with $H^1(B_2, \Z^{B_2}) \cong \Z^2$, but for which}  $H^1_c(\G_{B_2}, \shf{\Z})$ surjects onto $\Z[\frac{1}{2}]$.  
Hence,  $H^1_c(\G_{B_2}, \shf{\Z})$ is not finitely generated. 
Consequently, we cannot expect the map $H^n(\Lambda, A) \to H^n_c(\G_\Lambda, \shf{A})$ to be surjective in general.

Moreover, we have the following example of a 1-graph for which this map is not injective.  
This answers in the negative a conjecture of Kumjian, Pask, and Sims ({stated in the paragraph} immediately following Corollary 6.8 of \cite{kps-twisted}).

\begin{example}
\label{ex:not-isom}
Let $\Lambda$ be the path category of the $1$-graph with $\Lambda^0 := \{ v_n : n \in \N \}$ and edge set $\Lambda^1 := \{f, g \} \cup \{ e_1, e_2, \dots \}$, where  
$s(f) = s(g) = v_1$, $r(f) = r(g) = v_0$ and  $r(e_n) = v_n$, $s(e_n) = v_{n+1}$ for all $n = 1, 2, \dots$.
Then $\Lambda^\infty$ is a countably infinite space;  
there are two infinite paths  $x_\pm$ 
 {with range} $v_0$ and exactly one 
 with range $v_n$ for $n \ge 1$.  
Thus, 
\[ \G_\Lambda = \{ (x_+, 0, x_-), (x_-, 0, x_+)\} \cup \{ (x, n, \sigma^n(x)), (\sigma^n(x), -n, x): x \in \Lambda^\infty, \ n \in \N\}.\]
Let $c: \G_\Lambda \to \Z$ be a 1-cocycle; that is, an additive function. 
For any infinite path $x$, there is a unique $n \in \N$ such that $(x, -n, x_+) \in \G_\Lambda$.  Thus, setting 
\[b_c(x) := c(x, -n, x_+)\]
if $x \not= x_+$, and setting $b_c(x_+) = 0$, we have   $ c(x, m, y) = b_c(x) - b_c(y).$

In other words, every 1-cocycle on $\G_\Lambda$ is a coboundary, so $H^1(\G_\Lambda, \Z) = 0$.

However, ${H^1}(\Lambda, \Z^\Lambda) = \Z$ is generated by the cocycle $c$ given by 
\[ c(\lambda) = \begin{cases}
1, & \lambda = f e_1 \cdots e_\ell \text{ for some } \ell \in \N\\
0, & \text{ else.}
\end{cases}\]
The fact that $c$ has infinite order follows from the fact that $c(f) = 1$ but $c(g) = 0$, even though these edges have the same source and range.  In more detail, suppose  we have a function $b: \Lambda^0 \to \Z$ such that for all $\lambda \in \Lambda,\ n c (\lambda) =  b(r(\lambda)) - b(s(\lambda)).$
Then
\[ n = n c(f) =  b(v_0) - b(v_1) = n c(g) = 0.\]
\end{example}

\begin{rmk}
\label{rmk:ample-gpoid-sheaf}
Recall that an \'etale groupoid  with a basis of clopen sets is said to be an \emph{ample groupoid}.
For example,  $\G_\Lambda$  is ample for any  row-finite higher-rank graph  $\Lambda$  with no sources.

Given an \'etale groupoid $\G$ and a $\G$-sheaf $\mathscr{A}$, Theorem 2.7 of \cite{equiv-sheaf-coh} establishes that the group of isomorphism classes of groupoid extensions of the form $\mathscr{A} \to \H \to \G$
forms a group $T_\G(\mathscr{A})$ and the assignment $\mathscr{A}  \mapsto T_\G(\mathscr{A})$ is a half-exact covariant functor which is naturally isomorphic to the first derived functor of $Z_\G$
(the functor which assigns the group of continuous $\mathscr{A}$-valued 1-cocycles on $\G$ to $\mathscr{A}$).
When  $\G$ is an ample groupoid, every such extension has a continuous section and therefore is determined by a
continuous $\mathscr{A}$-valued 2-cocycle on $\G$.  It follows that 
\[H^2_c(\G, \mathscr{A}) \cong T_\G(\mathscr{A}).\]
Let $\mathscr{A}^0$ denote the same sheaf $\mathscr{A}$, regarded as an ordinary sheaf (with no $\G$-action).
Then by \cite[Theorem 3.7]{equiv-sheaf-coh} there is a long exact sequence
\begin{align*}
0 \to H^0(\G, \mathscr{A}) &\to H^0(\G\z, \mathscr{A}^0) \to Z_\G^0(\mathscr{A}) \to H^1(\G, \mathscr{A}) \to \\
&H^1(\G\z, \mathscr{A}^0) \to Z_\G^1(\mathscr{A}) \to H^2(\G, \mathscr{A}) \to  H^2(\G\z \mathscr{A}^0) \to \cdots.
\end{align*}
When $\G$ is ample, $\G\z$ has a basis of clopen sets, and thus $H^n(\G\z, \mathscr{A}^0) = 0$ for $n \ge 1$. 
It now follows by the above long exact sequence that
\[
Z_\G^{n}(\mathscr{A}) \cong H^{n+1}(\G, \mathscr{A}) \quad \text{ for all } n \geq 1.\] 
In the case $n=1$, the isomorphisms mentioned above combine to give us
\[
H^2_c(\G, \mathscr{A}) \cong T_\G(\mathscr{A}) \cong Z_\G^1(\mathscr{A}) \cong H^2(\G, \mathscr{A}).
\]
\end{rmk}

\textsc{Acknowledgments}: The second author would like to thank the first author and her colleagues at both the University of Colorado, Boulder and the Universit\"at M\"unster 
for their hospitality and support.  
He is also grateful to Wojcieh Szyma\'{n}ski for preliminary discussions relating to the cohomology of $k$-graphs and for his hospitality and support during a recent visit to the University of Southern Denmark.

\bibliographystyle{amsplain}
\bibliography{eagbib}
\end{document}